\documentclass[a4paper,10pt]{amsart}

\usepackage{amsmath,amssymb,amsfonts}
\usepackage{ifthen}
\usepackage{hyperref}
\usepackage[boxed]{algorithm2e}

\theoremstyle{plain}
\newtheorem{theorem}{Theorem}
\newtheorem{proposition}[theorem]{Proposition}
\newtheorem{lemma}[theorem]{Lemma}

\theoremstyle{definition}
\newtheorem{definition}[theorem]{Definition}

\newtheorem{remark}[theorem]{Remark}

\newcommand{\eps}{\varepsilon}

\DeclareMathOperator{\BV}{BV}
\DeclareMathOperator{\SBV}{SBV}
\DeclareMathOperator{\Id}{Id}
\DeclareMathOperator{\len}{len}
\newcommand{\field}[1]{\mathbb{#1}}
\newcommand{\R}{\field{R}}
\newcommand{\N}{\field{N}}
\newcommand{\Z}{\field{Z}}

\newcommand{\weaksto}{\rightharpoonup^{*}}

\DeclareMathOperator*{\esslim}{ess\,lim}

\newcommand{\inner}[3][n]{\SwitchBracketsizeLeft{#1}\LeftBracketSize\langle#2,#3\SwitchBracketsizeRight{#1}\RightBracketSize\rangle}
\newcommand{\abs}[2][n]{\SwitchBracketsizeLeft{#1}\LeftBracketSize\lvert#2\SwitchBracketsizeRight{#1}\RightBracketSize\rvert}
\newcommand{\norm}[2][n]{\SwitchBracketsizeLeft{#1}\LeftBracketSize\lVert#2\SwitchBracketsizeRight{#1}\RightBracketSize\rVert}
\newcommand{\set}[3][b]{\SwitchBracketsizeLeft{#1}\LeftBracketSize\{#2:#3\SwitchBracketsizeRight{#1}\RightBracketSize\}}

\newcommand{\NextScriptStyle}[1]{{\scriptstyle{#1}}}
\newcommand{\NextScriptScriptStyle}[1]{{\scriptscriptstyle{#1}}}
\newcommand{\NextTextStyle}[1]{{\textstyle{#1}}}
\newcommand{\NextDisplayStyle}[1]{{\displaystyle{#1}}}
\newcommand{\SwitchBracketsizeLeft}[1]{
  \ifthenelse{\equal{#1}{b}\OR\equal{#1}{big}}{\let\LeftBracketSize=\bigl}{
    \ifthenelse{\equal{#1}{B}\OR\equal{#1}{Big}}{\let\LeftBracketSize=\Bigl}{
      \ifthenelse{\equal{#1}{g}\OR\equal{#1}{bigg}}{\let\LeftBracketSize=\biggl}{
    \ifthenelse{\equal{#1}{G}\OR\equal{#1}{Bigg}}{\let\LeftBracketSize=\Biggl}{
      \ifthenelse{\equal{#1}{s}\OR\equal{#1}{small}}{\let\LeftBracketSize=\NextScriptStyle}{
        \ifthenelse{\equal{#1}{ss}}{\let\LeftBracketSize=\NextScriptScriptStyle}{
          \ifthenelse{\equal{#1}{t}\OR\equal{#1}{text}}{\let\LeftBracketSize=\NextTextStyle}{
        \ifthenelse{\equal{#1}{d}\OR\equal{#1}{display}}{\let\LeftBracketSize=\NextDisplayStyle}{
          \ifthenelse{\equal{#1}{a}\OR\equal{#1}{auto}}{\let\LeftBracketSize=\left}{
            \let\LeftBracketSize=\relax}}}}}}}}}}
\newcommand{\SwitchBracketsizeRight}[1]{
  \ifthenelse{\equal{#1}{b}\OR\equal{#1}{big}}{\let\RightBracketSize=\bigr}{
    \ifthenelse{\equal{#1}{B}\OR\equal{#1}{Big}}{\let\RightBracketSize=\Bigr}{
      \ifthenelse{\equal{#1}{g}\OR\equal{#1}{bigg}}{\let\RightBracketSize=\biggr}{
    \ifthenelse{\equal{#1}{G}\OR\equal{#1}{Bigg}}{\let\RightBracketSize=\Biggr}{
      \ifthenelse{\equal{#1}{s}\OR\equal{#1}{small}}{\let\RightBracketSize=\NextScriptStyle}{
        \ifthenelse{\equal{#1}{ss}}{\let\RightBracketSize=\NextScriptScriptStyle}{
          \ifthenelse{\equal{#1}{t}\OR\equal{#1}{text}}{\let\RightBracketSize=\NextTextStyle}{
        \ifthenelse{\equal{#1}{d}\OR\equal{#1}{display}}{\let\RightBracketSize=\NextDisplayStyle}{
          \ifthenelse{\equal{#1}{a}\OR\equal{#1}{auto}}{\let\RightBracketSize=\right}{
            \let\RightBracketSize=\relax}}}}}}}}}}

\DeclareMathOperator{\inn}{int}

\DeclareMathOperator{\AC}{AC}

\newcommand{\Lres}{\mathbin{\hbox{\vrule height6pt depth0pt \vrule height0.5pt depth0pt width5pt}}}

\date{\today}

\author{Markus Grasmair}

\address{Department of Mathematical Sciences, Norwegian University of Science and Technology (NTNU), 7491 Trondheim, Norway}
\email{markus.grasmair@ntnu.no}
\urladdr{https://www.ntnu.edu/employees/markus.grasmair}

\title[An SRV framework for BV curves]{A square root velocity framework for curves of bounded variation}


\begin{document}

\begin{abstract}
  The square root velocity transform is a powerful tool for the
  efficient computation of distances between curves.
  Also, after factoring out reparametrisations,
  it defines a distance between shapes
  that only depends on their intrinsic geometry
  but not the concrete parametrisation.
  Though originally formulated for smooth curves,
  the square root velocity transform and the
  resulting shape distance have been thoroughly analysed
  for the setting of absolutely continuous curves using
  a relaxed notion of reparametrisations.
  In this paper, we will generalise the square root velocity distance
  even further to a class of discontinuous curves.
  We will provide an explicit formula for the natural
  extension of this distance to curves of bounded variation
  and analyse the resulting quotient distance
  on the space of unparametrised curves.
  In particular, we will discuss the existence of optimal reparametrisations
  for which the minimal distance on the quotient space is realised.
\end{abstract}

\maketitle

\section{Introduction}

The mathematical analysis of shapes is a topic that is required in a large
number of different applications ranging from mathematical image processing
and inverse problems over
computational biology to different problems in computer vision;
see \cite{BauEslGra17,CelEslSchm16,EckHipHohSchuWar19,LiuSriZha11,SunMenSoaYez11}
to name but a few examples, see also~\cite{SriKla16} for a larger overview.
In order to perform tasks like classification and for studying the geometry
of datasets of shapes, it is necessary to define a suitable distance on
the \emph{shape space}.
Here we will focus on one-dimensional shapes,
that is, parametrised curves $c \colon I \to \R^d$, $I = [0,1]$,
where we identify two curves if they only
differ by a translation or a reparametrisation.
In this setting, a particularly useful distance is the
square root velocity (SRV) distance introduced in~\cite{MioSriJos07,Srietal11}.
Given a sufficiently smooth, regular curve $c$, we define its
square root velocity transform $q := R(c) = \dot{c}/\sqrt{\lvert \dot{c}\rvert}$.
Then, the distance between two curves $c_1$ and $c_2$ with square root velocity
transforms $q_1$ and $q_2$ is defined as $d(c_1,c_2) = \lVert q_1 - q_2 \rVert_{L^2}$.
It is also possible to regard the space of all smooth regular curves
as a manifold with the Riemannian structure inherited from $L^2(I;\R^d)$
via the mapping $R$.
This differential geometric point of view of shape analysis has been studied and
discussed for instance in~\cite{BauBruMic14,BauChaKlaLeB20,MioSriJos07,SriKla16}.

With this definition, one does not yet obtain a distance on the shape space,
as $d(c_1,c_2)$ depends on the parametrisation of the curves $c_1$ and $c_2$.
For that, one needs to consider instead the quotient distance
\begin{equation}\label{eq:dSintro}
  d^S([c_1],[c_2]) = \inf_{\varphi_1,\,\varphi_2} d(c_1\circ \varphi_1,c_2\circ\varphi_2),
\end{equation}
where the infimum is taken over all orientation preserving smooth diffeomorphisms of
the unit interval $I$.
Various approaches for the efficient numerical solution of (discretisations of)
this optimisation problem have been suggested, ranging from
gradient based optimisation methods~\cite{HuaGalSriAbs16}
over dynamical programming~\cite{DogBerHag15,MioSriJos07}
and the reformulation as a Hamilton--Jacobi--Bellman equation~\cite{WoiGra22}
to machine learning methods~\cite{HarSukChaKlaBau21,NunJos20}.
There exists also an analytic algorithm for the case where the curves
$c_1$ and $c_2$ are piecewise linear~\cite{LahRobKla15}.

In addition to providing a distance between shapes,
the actual solutions $(\bar{\varphi}_1,\bar{\varphi}_2)$ of the
optimisation problem~\eqref{eq:dSintro} can be used to define
correspondences, or best matches, between the curves $c_1$ and $c_2$:
The best match on $c_2$ for the point $c_1(t)$ is the point
$c_2(\bar{\varphi}_2(\bar{\varphi}_1^{-1}(t)))$.
To that end, however, it is necessary that the optimisation problem~\eqref{eq:dSintro}
actually admits a solution.
In~\cite{Bru16}, it has been shown that this is the case
provided that the two curves $c_1$ and $c_2$ are continuously
differentiable and the reparametrisations $\varphi_1$ and $\varphi_2$ are
allowed to be merely absolutely continuous and non-decreasing
instead of being smooth and strictly increasing.
At the same time, an example of two Lipschitz curves was provided,
where the infimum in~\eqref{eq:dSintro} is not attained.
A closer inspection of the results of~\cite{Bru16}, however, reveals
that a convex relaxation of the problem~\eqref{eq:dSintro}
admits its minimum for arbitrary absolutely continuous curves $c_1$ and $c_2$.
Using this relaxed distance measure, one can thus generalise
shape distances to absolutely continuous curves and still
obtain best matches between arbitrary curves.
Because the reparametrisations are merely non-decreasing
and not necessarily strictly increasing, it is possible, though,
that a single point on one of the curves corresponds to a whole
line segment on the other curve.

\medskip

In this paper, we will generalise the analysis of~\cite{Bru16}
to discontinuous curves, or, more specifically, to curves of bounded variation.
There it is no longer possible to define the square root velocity
transform $R(c)$ in a meaningful manner, as this would involve taking
the square root of the derivative of the curve $c$, which is a Radon measure.
Instead, we work directly with the SRV distance and 
provide an explicit formula for a relaxation of this distance
from absolutely continuous curves to curves of bounded variation,
see Theorem~\ref{th:relax} below.
The main tool here is a generalisation of the Reshetnyak continuity
and lower semi-continuity theorems~\cite{Res68}.
Next, we discuss how the relaxed SRV distance can be
used for defining a shape distance on curves of bounded variation.
The main challenge here is the fact that a composition
of a discontinuous curve with a non-decreasing
reparametrisation is not necessarily well-defined:
If the reparametrisation $\varphi$ maps a whole interval $[a,b]$
to a single point $t$ where the curve $c$ has a jump,
then it is \emph{a--priori} not clear how the composition
$c\circ\varphi$ should be interpreted on that interval.
In Definition~\ref{de:repar} we provide such an interpretation,
which we show to be natural in the context we are working in,
see Proposition~\ref{pr:relaxD}.
Finally, we consider the particular setting of special curves
of bounded variation, where we show that our approach gives
rise to a shape distance, for which optimal matchings
exists for each pair of curves, see Theorem~\ref{th:main2}.

In Section~\ref{se:AC} we will recall the main definitions
and results concerning the SRV transform for absolutely
continuous curves, which are relevant for this paper.
The generalisation to curves and shapes of bounded variation
is presented in Section~\ref{se:BV}.
Finally, the proofs of all the results are collected in
Sections~\ref{se:pfthm1} and~\ref{se:pfshape}.

\section{The SRV framework for absolutely continuous curves}\label{se:AC}

In the following, we will provide a brief introduction
into the square root velocity framework for absolutely
continuous curves following the results of~\cite{Bru16}.

\subsection{Square root velocity transform}
Denote by $I = [0,1]$ the unit interval
and by $\AC(I;\R^d)$ the space of absolutely continuous
curves in $\R^d$. Moreover, let $\AC_0(I;\R^d)$ the subspace
of absolutely continuous curves satisfying $c(0) = 0$.
The square root velocity-transform (SRVT)
of a curve $c \in \AC(I;\R^d)$ is defined as
$R \colon \AC(I;\R^d) \to L^2(I;\R^d)$,
\begin{equation}\label{eq:SRVT}
  R(c) = \frac{\dot{c}}{\sqrt{\abs{\dot{c}}}}.
\end{equation}
Here the fraction $\dot{c}/\sqrt{\abs{\dot{c}}}$ is set to be zero
at points where $\dot{c} = 0$.
The mapping $R$ is a bijection from $\AC_0(I;\R^d)$ to $L^2(I;\R^d)$
with inverse
\[
  R^{-1}(q)(x) = \int_0^x q\abs{q}\,dy.
\]
Moreover, the SRVT is norm-preserving in the sense that
\[
  \norm{R(c)}_{L^2}^2 = \norm{\dot{c}}_{L^1} = \len(c).
\]

Given two curves $c_1$, $c_2 \in \AC_0(I;\R^d)$, we
define their (squared) SRV distance as
\[
  d(c_1,c_2)^2 = \norm{R(c_1)-R(c_2)}_{L^2}^2.
\]
This can be rewritten as
\[
  d(c_1,c_2)^2 = \len(c_1) + \len(c_2) - 2 S(c_1,c_2)
\]
with
\[
  S(c_1,c_2) = \langle R(c_1),R(c_2)\rangle_{L^2}
  = \int_I \Bigl\langle \frac{\dot{c}_1}{\abs{\dot{c}_1}},\frac{\dot{c}_2}{\abs{\dot{c}_2}}\Bigr\rangle
  \sqrt{\abs{\dot{c}_1}\abs{\dot{c}_2}}\,dx.
\]

\subsection{Shape space distance}
  
With the definition above one obtains a distance on
the space of absolutely continuous curves.
However, we are also interested in the case where one
identifies curves if they are equal up to parametrisation.
Following~\cite{Bru16}, we define
\[
  \Gamma := \set{\gamma \in \AC(I;I)}{\gamma(0)=0,\,\gamma(1)=1,\, \gamma' > 0 \text{ a.e.\,}},
\]
the set of all absolutely continuous reparametrisations of the unit interval.
Moreover, we define
\[
  \bar{\Gamma} := \set{\gamma \in \AC(I;I)}{\gamma(0)=0,\,\gamma(1)=1,\, \gamma' \ge 0 \text{ a.e.\,}}.
\]
Then $\bar{\Gamma}$ is the closure of $\Gamma$ in $\AC(I;I)$
(with respect to the norm topology).

We say that two curves $c_1$, $c_2 \in \AC_0(I;\R^d)$ are equivalent,
if there exist $\varphi_1$, $\varphi_2 \in \bar{\Gamma}$ and a curve
$\hat{c} \in \AC_0(I;\R^d)$ such that $c_1 = \hat{c}\circ\varphi_1$
and $c_2 = \hat{c}\circ\varphi_2$.
It has been shown in~\cite[Prop.~12]{Bru16} that this defines
an equivalence relation $\sim$ on $\AC_0(I;\R^d)$.
In the following, we will denote the equivalence class of a curve $c$ by $[c]$.
Moreover, the quotient space of unparametrised curves is denoted
by $B(I;\R^d) = \AC_0(I;\R^d)/\sim$.
One can show that
\[
  B(I;\R^d) = \bigl\{c\circ\bar{\Gamma} : c \in \AC_0(I;\R^d),\,\dot{c}\neq 0\text{ a.e.}\bigr\} \cup \{0\}.
\]
Moreover, two curves are equivalent, if and only if they
have the same constant speed parametrisation.

On the quotient space $B(I;\R^d)$ we now consider the induced distance
\[
  d^{\mathcal{S}}([c_1],[c_2]) = \inf_{\varphi_1,\,\varphi_2 \in \bar{\Gamma}} d(c_1\circ\varphi_1,c_2\circ\varphi_2).
\]
Note here that, for $c \in \AC(I;\R^d)$ and $\varphi \in \bar{\Gamma}$,
the composition $c\circ\varphi$ is again absolutely continuous,
since $\varphi$ is a non-decreasing function~\cite[Prop.~225C]{Fre03}.

Since the length of a curve is invariant under reparametrisation,
we can also write
\begin{equation}\label{eq:distAC}
  d^S([c_1],[c_2])^2 = \len(c_1) + \len(c_2) - \sup_{\varphi_1,\,\varphi_2 \in \bar{\Gamma}} S(c_1\circ\varphi_1,c_2\circ\varphi_2)
\end{equation}
with
\[
  S(c_1\circ\varphi_1,c_2\circ\varphi_2) =
  \int_I \inner[g]{\frac{\dot{c}_1 \circ \varphi_1}{\abs{\dot{c}_1\circ \varphi_1}}}
  {\frac{\dot{c}_2\circ \varphi_2}{\abs{\dot{c}_2\circ\varphi_2}}}
  \sqrt{\abs{\dot{c}_1\circ\varphi_1}\abs{\dot{c}_2\circ\varphi_2}\varphi_1'\varphi_2'}\,dx.
\]
It has been shown in~\cite[Prop.~15]{Bru16} that the supremum in~\eqref{eq:distAC}
is attained for some $\varphi_1$, $\varphi_2 \in \bar{\Gamma}$ provided that $c_1$, $c_2 \in C^1(I;\R^d)$
satisfy $\dot{c}_i \neq 0$ almost everywhere.

Note that the functional $S(c_1,c_2)$ is invariant under simultaneous reparametrisations
of the curves $c_1$ and $c_2$, that is,
$S(c_1\circ\varphi,c_2\circ\varphi) = S(c_1,c_2)$ for all $\varphi \in \bar{\Gamma}$.
This property is crucial for the geometric properties
of the quotient space $B(I;\R^d)$,
and is also often exploited in numerical discretisations of~\eqref{eq:distAC}.

\subsection{Scale invariant SRV distance}

It is also possible to define a scale invariant SRV distance.
Here one defines, for $c \in \AC_0(I;\R^d)\setminus\{0\}$,
\[
  \tilde{R}(c) = \frac{1}{\sqrt{\len(c)}} \frac{\dot{c}}{\sqrt{\abs{\dot{c}}}} = R(c/\len(c)).
\]
This gives a mapping from $\AC_0(I;\R^d) \setminus\{0\}$ to the
unit sphere in $L^2(I;\R^d)$.
The corresponding scale invariant SRV distance between two non-zero curves
is therefore defined as the spherical distance between their scale invariant SRVTs,
that is,
\begin{multline*}
  \tilde{d}(c_1,c_2)^2 = \arccos\bigl(\langle R(c_1/\len(c_1)),R(c_2/\len(c_2))\rangle_{L^2}\bigr)\\
  = \arccos\bigl(S(c_1/\len(c_1),c_2/\len(c_2))\bigr).
\end{multline*}
Moreover, we can define a scale invariant shape distance by
\begin{multline*}
  \tilde{d}^S([c_1],[c_2])^2 = \inf_{\varphi_1,\,\varphi_2 \in \bar{\Gamma}} \tilde{d}(c_1\circ\varphi_1/\len(c_1),c_2\circ\varphi_2/\len(c_2))^2\\
  = \arccos\Bigl(\sup_{\varphi_1,\,\varphi_2 \in \bar{\Gamma}} S(c_1\circ\varphi_1/\len(c_1),c_2\circ\varphi_2/\len(c_2))\Bigr).
\end{multline*}
In this article, we will focus on the unscaled variant of the SRV distance.
However, all our results are based solely on the properties of
the function $S$, which is central both for the scaled and the unscaled variants.
Thus all our results hold \emph{mutatis mutandis} also for the scaled SRV framework.

\section{Generalisation to BV curves}\label{se:BV}

We now want to generalise the SRV distance to discontinuous curves,
or, more specifically, to curves of bounded variation.
Note here that we will not generalise the SRVT itself,
since the definition of $R(c)$ involves the square root of the derivative
of $c$, which is a measure if $c \in \BV(I;\R^d)$ is a general function of bounded variation.
However, we will see that the resulting distance can still be defined.

\subsection{Curves of bounded variation}

In the following, we collect some results concerning functions of
bounded variation that will be needed throughout the paper.
For more details, we refer to~\cite[Sec.~3.2]{AmbFusPal00}.

For every $c \in \BV(I;\R^d)$
and $x \in I$ the one-sided essential limits $c^\ell(x) := \esslim_{y \to x^-} c(y)$
and $c^r(x) := \esslim_{y \to x^+}$ are well-defined.
Thus we can define the subspace
\[
  \BV_0(I;\R^d) := \bigl\{ c \in \BV(I;\R^d) : c^{r}(0) = 0\bigr\}.
\]
For $c \in \BV_0(I;\R^d)$ and every $x \in I$ we have that
$c^{\ell}(x) = Dc((0,x))$ and $c^{r}(x) = Dc((0,x])$.

We say that a pointwise defined function $\tilde{c} \colon I\to\R^d$
is a \emph{good representative} of $c \in \BV(I;\R^d)$,
if for all $x \in I$ we have $\tilde{c}(x) \in [c^\ell(x),c^r(x)]$.
That is, $\tilde{c}(x) = \esslim_{y\to x} c(x)$ whenever $c$ is essentially continuous,
and $\tilde{c}(x)$ lies on the line segment from $c^\ell(x)$ to $c^r(x)$
whenever $c$ has a jump at $x$.
We will always identify $c$ with any of its good representatives.
Finally, for good representatives the measure theoretic total variation of $c$
coincides with the pointwise total variation, that is,
\[
  \abs{Dc}(I) = \sup\Bigl\{\sum_k \abs{c(x_{k+1}) - c(x_k)} : 0 < x_1 < x_2 < \ldots < x_N < 1\Bigr\}.
\]

In a slight abuse of notation, we denote by $\len(c) := \abs{Dc}(I)$
the length of a discontinuous curve, including all of its jumps.
We say that a sequence $\{c^{(k)}\}_{k\in\N} \subset \BV(I;\R^d)$
converges \emph{strictly} to $c \in \BV(I;\R^d)$,
denoted
\[
  c^{(k)} \to^s c,
\]
if $\lVert c^{(k)} - c\rVert_{L^1} \to 0$ and $\len(c^{(k)}) \to \len(c)$.
Equivalently, we have that $c^{(k)}\to^s c$, if and only if
$c^{(k)}(0) \to c(0)$, $Dc^{(k)} \weaksto Dc$ in the sense of weak convergence
of Radon mesures, and $\len(c^{(k)}) \to \len(c)$.

Assume that $c \in \BV(I;\R^d)$.
Then its weak derivative $Dc \in \mathcal{M}(I;\R^d)$ can be decomposed as
\[
  Dc = \dot{c}\Lres \mathcal{L}^1 + D^j c + D^c c,
\]
where $\dot{c}$ is the classical derivative of $c$ (which exists almost everywhere),
$D^j c$ is a purely atomic measure---the jump part of $Dc$---,
and $D^c c$ is a non-atomic, singular measure---the Cantor part of $Dc$.
We denote for $c \in \BV(I;\R^d)$
by
\[
  \Sigma(c) := \set{x \in I}{c^{\ell}(x) \neq c^{r}(x)}
\]
the jump set of $c$ and define
\[
  [c](x) := c^{r}(x)-c^{\ell}(x)
  \qquad\text{ for } x \in \Sigma(c).
\]
Then the jump part $D^j c$ of $Dc$ can be written as
\[
  D^j c = \sum_{x \in \Sigma(c)} [c](x)\,\delta_x
\]
with $\delta_x$ denoting the Dirac delta centered at $x$.

Furthermore, we can decompose every function $c \in \BV_0(I;\R^d)$
uniquely as
\[
  c = c^{(a)} + c^{(j)} + c^{(c)}
\]
such that $c^{(a)}$ is absolutely continuous and $(c^{(a)})' = \dot{c}$,
$Dc^{(j)} = D^j c$, and $Dc^{(c)} = D^c c$.
The set of functions $c$ where the Cantor part $c^{(c)}$
equals zero is denoted $\SBV(I;\R^d)$.

\subsection{Extension of the SRV distance}

We define
\[
  \hat{S}(c_1,c_2)
  := \sup\Bigl\{\limsup_k S(c_1^{(k)},c_2^{(k)}) : c_i^{(k)} \in \AC(I;\R^d),\ c_i^{(k)}\to^s c_i\Bigr\}
\]
and
\[
  \hat{d}(c_1,c_2) = \len(c_1) + \len(c_2) - 2\hat{S}(c_1,c_2).
\]
That is, $\hat{d}$ is the strictly lower semi-continuous hull
of the function that is equal to $d$ on $\AC(I;\R^d)^2$ and equal to $+\infty$ else.

Our first main theorem provides an explicit expression for $\hat{S}$,
and consequently also for $\hat{d}$.

\begin{theorem}\label{th:relax}
  Assume that $c_1$, $c_2 \in \BV(I;\R^d)$. Then
  \begin{multline*}
    \hat{S}(c_1,c_2) \\
    = \int_I \inner[B]{\frac{dDc_1}{\abs{dDc_1}}}{\frac{dDc_2}{\abs{dDc_2}}}^+
    \sqrt{\frac{d\abs{Dc_1}}{d(\abs{Dc_1}+\abs{Dc_2})}\frac{d\abs{Dc_2}}{d(\abs{Dc_1}+\abs{Dc_2})}}\,d(\abs{Dc_1}+\abs{Dc_2}).
  \end{multline*}
  Here the integrand is set to zero at points where either
  $d\abs{Dc_1}/d(\abs{Dc_1}+\abs{Dc_2}) = 0$ or $d\abs{Dc_2}/d(\abs{Dc_1}+\abs{Dc_2})=0$.
  Moreover $(\cdot)^+ := \max\{\cdot,0\}$ denotes the positive part of the argument.

  In particular, if $c_1$, $c_2 \in \SBV(I;\R^d)$, then
  \begin{multline*}
    \hat{S}(c_1,c_2) = \int_I \inner[g]{\frac{\dot{c}_1}{\abs{\dot{c}_1}}}{\frac{\dot{c}_2}{\abs{\dot{c}_2}}}^+
    \sqrt{\abs{\dot{c}_1}\abs{\dot{c}_2}}\,dx \\
    + \sum_{x \in \Sigma(c_1)\cap \Sigma(c_2)} \inner[g]{\frac{[c_1](x)}{\abs{[c_1](x)}}}{\frac{[c_2](x)}{\abs{[c_2](x)}}}^+
    \sqrt{\abs{[c_1](x)}\,\abs{[c_2](x)}}.
  \end{multline*}
\end{theorem}

\begin{proof}
  See Section~\ref{se:pfthm1}.
\end{proof}

\subsection{Reparametrisations of BV curves}

Again, we are interested in the case of distances modulo reparametrisations.
However, in a setting with discontinuous curves $c$ and reparametrisations $\varphi$
that are merely non-decreasing, but not necessarily strictly increasing,
the expression $c\circ\varphi$ does not always make sense:
If $\varphi$ is constant on a non-trivial interval, say $\varphi(x) = y$
for all $x \in [a,b]$ with $b > a$ and $c$ has a jump at $y$,
then $c \circ \varphi$ is not well-defined on the interval $[a,b]$.
Thus we have to use a generalised definition of reparametrisations.

\begin{definition}\label{de:repar}
  Let $c \in \BV(I;\R^d)$ and $\varphi \in \bar{\Gamma}$.
  We define
  \begin{multline*}
    [c,\varphi] := \Bigl\{ g \in \BV(I;\R^d) :
    g(x) \in \bigl[c^\ell\bigl(\varphi(x)\bigr),c^r\bigl(\varphi(x)\bigr)\bigr] \text{ for all } x \in I\\
    \text{ and } \len(c) = \len(g)\Bigr\}.
  \end{multline*}
\end{definition}

Here $\bigl[c^\ell\bigl(\varphi(x)\bigr),c^r\bigl(\varphi(x)\bigr)\bigr] \subset \R^d$
denotes the line segment from $c^\ell\bigl(\varphi(x)\bigr)$ to $c^r\bigl(\varphi(x)\bigr)$.
In particular, if $x \in I$ is such that $c$ is continuous at
$\varphi(x)$ and $g \in [c,\varphi]$, then $g(x) = c(\varphi(x))$ and $g$ is continuous
at $x$.
Moreover, note that $[c,\varphi]$ consists of the single element $c\circ\varphi$
if either $c$ is continuous or $\varphi$ injective.

With a slight abuse of notation, we thus can define the distance
between the sets $[c_1,\varphi_1]$ and $[c_2,\varphi_2]$ as
\[
  \hat{d}([c_1,\varphi_1],[c_2,\varphi_2])
  := \inf_{g_i \in [c_i,\varphi_i]} \hat{d}(g_1,g_2)
  = \len(c_1)+\len(c_2) - \sup_{g_i \in [c_i,\varphi_i]} \hat{S}(g_1,g_2).
\]
Our next result shows that the same distance function is obtained as the
lower semi-continuous extension of the mapping
$(c_1,c_2,\varphi_1,\varphi_2) \mapsto d(c_1\circ\varphi_1,c_2\circ\varphi_2)$
with respect to strict convergence of the curves $c_i$ and uniform convergence
of the reparametrisations $\varphi_i$.

\begin{proposition}\label{pr:relaxD}
  Define the functional $D \colon \BV(I;\R^d)^2 \times \bar{\Gamma}^2 \to \R\cup\{+\infty\}$,
  \[
    D(c_1,c_2;\varphi_1,\varphi_2) :=
    \begin{cases}
      d(c_1\circ\varphi_1,c_2\circ\varphi_2) & \text{ if } c_i \in \AC(I;\R^d) \text{ and } \varphi_i \in \Gamma,\\
      +\infty &\text{ else.}
    \end{cases}
  \]
  Then the lower semi-continuous hull of $D$
  with respect to strict convergence on $\BV(I;\R^d)$ and uniform convergence
  on $\bar{\Gamma}$ is the functional
  \[
    \hat{D}(c_1,c_2;\varphi_1,\varphi_2)
    = \inf_{g_i \in [c_i,\varphi_i]} \hat{d}(g_1,g_2)
    = \hat{d}([c_1,\varphi_1],[c_2,\varphi_2]).
  \]
\end{proposition}

\begin{proof}
  See Section~\ref{se:relaxD}.
\end{proof}

The next result shows that the resulting distance function is invariant under simultaneous
reparametrisations:

\begin{proposition}\label{pr:Dinvar}
  Assume that $c_1$, $c_2 \in \BV(I;\R^d)$ and $\varphi_1$, $\varphi_2 \in \bar{\Gamma}$.
  Then we have for all $\psi \in \bar{\Gamma}$ that
  \[
    \hat{d}([c_1,\varphi_1\circ\psi],[c_2,\varphi_2\circ\psi])
    = \hat{d}([c_1,\varphi_1],[c_2,\varphi_2]).
  \]
\end{proposition}

\begin{proof}
  See Section~\ref{se:relaxD}.
\end{proof}

\subsection{A shape distance on SBV curves}\label{se:shapedist}

We now consider the particular setting where the involved curves
are special functions of bounded variation.

Assume that $c \in \SBV(I;\R^d)\setminus\{0\}$. Then we can construct an ``equivalent''
function $G(c) \in \AC(I;\R^d)$ in the following way:
Define the function $\xi\colon I \to I$,
\[
  \xi(x) = \frac{\abs{D^s c}(0,x)}{2\len(c)} + (1-\alpha) x,
  \qquad\qquad
  \text{ with }
  \alpha = \frac{\abs{D^s c}(I)}{2\len(c)}.
\]
Since $0 \le \alpha \le 1/2$, it follows that
$\xi$ is a strictly increasing function with $\xi(0) = 0$, $\xi(1) = 1$,
and $\Sigma(\xi) = \Sigma(c)$.
Denote now by $\zeta \colon I\to I$ the (unique) non-decreasing left inverse
of $\xi$. Then $\zeta \colon I \to I$ is Lipschitz continuous with
Lipschitz constant $1/(1-\alpha) \le 2$.
Now define the function $G(c) \in \AC(I;\R^d)$ by setting
$G(c)(x) = c(\zeta(x))$ if $c$ is continuous at $\zeta(x)$, and
\[
  G(c)(x) = c^{\ell}(\zeta(x)) + \frac{x-\xi^{\ell}(\zeta(x))}{\xi^{r}(\zeta(x))-\xi^{\ell}(\zeta(x))} \bigl(c^{r}(\zeta(x))-c_i^{\ell}(\zeta(x))\bigr)
\]
if $c$ and thus also $\xi$ is discontinuous at $\zeta(x)$.
That is, the jumps of $c$ are replaced by a linear interpolation between
the left and right limits of $c$ at the jump points.
In particular, we have that $G(c) \in [c,\zeta]$.

\begin{definition}\label{de:equiv}
  We say that two curves $c_1$, $c_2 \in \SBV(I;\R^d)$ are equivalent,
  denoted $c_1 \sim c_2$, if the curves $G(c_1)$ and $G(c_2)$
  have the same constant speed parametrisation.
  By $[c]$ we denote the equivalence class of the curve $c$.

  Moreover, we define the \emph{shape distance}
  \[
    \hat{d}^S([c_1],[c_2]) = \inf_{g_i \in [c_i]} \hat{d}(g_1,g_2)
    = \len(c_1) + \len(c_2) - \sup_{g_i \in [c_i]} \hat{S}(g_1,g_2)
  \]
  on the set of equivalence classes with respect to $\sim$.
\end{definition}

The following result shows that the distance between the shapes $[c_1]$ and $[c_2]$
can be computed by minimising the curve distance over all reparametrisations of
$c_1$ and $c_2$ in the sense of Definition~\ref{de:repar}.

\begin{theorem}\label{th:main2}
  Assume that $c_1$, $c_2 \in \SBV(I;\R^d)$ satisfy $\dot{c}_i(x) \neq 0$ for a.e.~$x\in I$. Then
  \begin{equation}\label{eq:sdistSBV}
    \hat{d}^S([c_1],[c_2]) = \inf_{\varphi_i \in \bar{\Gamma}} \hat{d}([c_1,\varphi_1],[c_2,\varphi_2])
    = \inf_{\psi_i \in \bar{\Gamma}}\hat{d}(G(c_1)\circ\psi_1,G(c_2)\circ\psi_2).
  \end{equation}
  Moreover, the infima in~\eqref{eq:sdistSBV} are attained
  at some $\bar{\varphi}_1$, $\bar{\varphi}_2 \in \bar{\Gamma}$
  and $\bar{\psi}_1$, $\bar{\psi}_2 \in \bar{\Gamma}$.
\end{theorem}

\begin{proof}
  See Section~\ref{se:pfmain2}.
\end{proof}

In addition, this result shows that it is possible to compute
optimal reparametrisations for arbitrary SBV curves.
In particular, it is possible to define optimal matchings
between SBV curves.

\section{Extension of the SRV distance to BV curves}\label{se:pfthm1}

In this section, we will prove Theorem~\ref{th:relax}, which
provides an explicit form for the extension of $S$ to curves of bounded variation.
In order to do so, we will make use of the fact that $S$ only depends
on the derivatives of the involved curves.
This allows us to reformulate Theorem~\ref{th:relax} as a result
concerning the extension of integral functionals from $L^1(I;\R^d)$ to $\mathcal{M}(I;\R^d)$,
the space of $\R^d$-valued finite Radon measures on $I$.
To that end, we define the functional $F \colon L^1(I;\R^d)^2 \to \R$,
\begin{equation}\label{eq:Fdef}
  F(u,v) := -\int_I \inner[B]{\frac{u}{\abs{u}}}{\frac{v}{\abs{v}}}
  \sqrt{\abs{u}\abs{v}}\,dx,
\end{equation}
and consider its extension $\hat{F}$ to $\mathcal{M}(I;\R^d)$, defined as
\begin{multline*}
  \hat{F}(\mu,\nu) = \inf\Bigl\{\liminf_k F(u^{(k)},v^{(k)}) : u^{(k)}\mathcal{L}^1 \weaksto \mu,\ v^{(k)}\mathcal{L}^1 \weaksto \nu,\\
  \lVert u^{(k)}\rVert_{L^1} \to \lvert\mu\rvert(I),\ \lVert v^{(k)}\rVert_{L^1} \to \lvert\nu\rvert(I)\Bigr\}
\end{multline*}
for $\mu$, $\nu \in \mathcal{M}(I;\R^d)$.
Then
\[
  S(c_1,c_2) = -F(\dot{c}_1,\dot{c}_2)
\]
for all $c_1$, $c_2 \in \AC(I;\R^d)$. Moreover, due to the definition
of strict convergence on $\BV(I;\R^d)$ and since we can, up to translations, identify
a curve with its derivative, we have that
\[
  \hat{S}(c_1,c_2) = - \hat{F}(Dc_1,Dc_2)
\]
for all $c_1$, $c_2 \in \BV(I;\R^d)$.
Thus it is sufficient to derive an explicit formula for the functional $\hat{F}$.

To do so, we will prove a generalisation of Reshetnyak's continuity and
lower semi-continuity theorems \cite{Res68}.
These theorems essentially state that a positively homogeneous integral functional
on $\mathcal{M}(I;\R^d)$ is continuous with respect to strict convergence,
and that it is weakly$^*$ lower semi-continuous,
if and only if the integrand is convex.
This result is not immediately applicable to our situation, as we are
dealing with a functional depending on two measures,
and we require lower semi-continuity with respect to separate strict convergence.
Thus, in Section~\ref{se:lsc}, we will formulate a generalisation of Reshetnyak's
continuity theorem that provides a lower bound for the functional $F$.
Then, in Section~\ref{se:relaxBV}, we will show that this lower bound is
actually sharp.
Finally, we will conclude the proof of Theorem~\ref{th:relax} in Section~\ref{se:relaxS}.

\subsection{Lower semi-continuity of integral functionals}\label{se:lsc}

Let $\Omega \subset \R^n$ be open and bounded, and let
\[
f \colon \bar{\Omega} \times \Sigma \times S^{d-1} \times S^{d-1} \to \R,
\]
where $\Sigma$ denotes the one-dimensional unit simplex and $S^{d-1}$ the $d-1$-dimensional unit sphere.
For simplicity, we identify $\Sigma$ with the interval $[0,1]$.
Assume that $f$ is lower semi-continuous and bounded and that
for every $x \in \Omega$ the mappings
\[
(\xi,\zeta) \mapsto f(x,0,\xi,\zeta)
\]
and
\[
(\xi,\zeta) \mapsto f(x,1,\xi,\zeta)
\]
are constant.

Then we can define the functional $F \colon \mathcal{M}(\Omega,\R^d)^2 \to \R$,
\[
F(\mu,\nu) = \int_\Omega f\Bigl(x,\frac{d\abs{\mu}}{d(\abs{\mu}+\abs{\nu})},\frac{d\mu}{d\abs{\mu}},\frac{d\nu}{d\abs{\nu}}\Bigr)\, d(\abs{\mu}+\abs{\nu}).
\]
Since $f(x,\tau,\cdot,\cdot)$ is assumed to be constant for $\tau \in \{0,1\}$,
the integrand is independent of the choice of
$d\mu/d\abs{\mu}$ and $d\nu/d\abs{\nu}$ outside of the supports of
$\abs{\mu}$ and $\abs{\nu}$, respectively, and thus the integral is well-defined.

\begin{theorem}\label{th:lsc}
Assume that $f$ is lower semi-continuous and bounded and that the mapping
\[
\tau \mapsto f(x,\tau,\xi,\zeta)
\]
is convex for every $x \in \Omega$ and $\xi$, $\zeta \in S^{d-1}$.
Then $F$ is lower semi-continuous with respect to strict convergence in both components.
That is, assume that $\mu_k \weaksto \mu$ and $\nu_k \weaksto \nu$ in
$\mathcal{M}(\Omega;\R^d)$ such that $\abs{\mu_k}(\Omega) \to \abs{\mu}(\Omega)$
and $\abs{\nu_k}(\Omega) \to \abs{\nu}(\Omega)$. Then
\[
F(\mu,\nu) \le \liminf_k F(\mu_k,\nu_k).
\]
\end{theorem}

\begin{proof}
  We follow the proof of Reshetnyak's continuity theorem as presented
  in~\cite[Thm.~10.3]{Rin18} (see also~\cite[Thm.~2.38, 2.39]{AmbFusPal00}).
  
  For simplicity, we write $m_k = d\mu_k/d\abs{\mu_k}$, $n_k = d\nu_k/d\abs{\nu_k}$,
  $t_k = d\abs{\mu_k}/d(\abs{\mu_k}+\abs{\nu_k})$,
  and similarly $m=d\mu/d\abs{\mu}$, $n=d\nu/d\abs{\nu}$, $t = d\abs{\mu}/d(\abs{\mu}+\abs{\nu})$.
  Next we define the measures $\sigma_k$ on $\Omega \times \Sigma \times S^{d-1} \times S^{d-1}$ by 
  \[
    \sigma_k = (\abs{\mu_k} + \abs{\nu_k}) \otimes (\delta_{t_k(x)} \otimes \delta_{m_k(x)} \otimes \delta_{n_k(x)}),
  \]
  that is,
  \[
    \int_{\Omega \times \Sigma \times S^{d-1} \times S^{d-1}} \varphi(x,\tau,\xi,\zeta)\, d\sigma_k
    = \int_\Omega \varphi\bigl(x,t_k(x),m_k(x),n_k(x)\bigr)\,d(\abs{\mu_k} + \abs{\nu_k})
  \]
  for every $\varphi \in C_0(\Omega\times\Sigma\times S^{d-1}\times S^{d-1})$.
  Since the measures $\abs{\mu_k}$ and $\abs{\nu_k}$ are uniformly bounded, it follows that the sequence $\sigma_k$ is bounded as well.
  After possibly passing to a subsequence, we may assume without loss of generality that $\sigma_k \weaksto \sigma$ for some $\sigma \in \mathcal{M}(\Omega\times \Sigma\times S^{d-1} \times S^{d-1})$.

  Denoting by $\pi \colon \Omega \times \Sigma \times S^{d-1} \times S^{d-1} \to \Omega$ the projection onto the first component, we obtain that $\pi_{\#} \sigma_k = \abs{\mu_k} + \abs{\nu_k}$.
  Since $\mu_k$ and $\nu_k$ converge strictly to $\mu$ and $\nu$,
  it follows that $\abs{\mu_k}$ and $\abs{\nu_k}$ converge weakly$^*$ to $\abs{\mu}$ and $\abs{\nu}$ (see~\cite[Cor.~10.2]{Rin18}). Thus
  $\pi_{\#} \sigma_k \weaksto \pi_{\#} \sigma = \abs{\mu} + \abs{\nu}$.

  Now (see~\cite[Thm.~4.4]{Rin18}) there exists a weakly$^*$ measurable family $\rho_x \in \mathcal{M}(\Sigma \times S^{d-1} \times S^{d-1})$ such that $\rho_x(\Sigma \times S^{d-1} \times S^{d-1}) = 1$ and $\sigma = (\abs{\mu} + \abs{\nu}) \otimes \rho_x$.

  Let now $\psi \in C_0(\Omega)$ be continuous and define $\varphi \in C_0(\Omega \times \Sigma \times S^{d-1} \times S^{d-1};\R^d)$,
  \[
    \varphi(x,\tau,\xi,\zeta) = \psi(x) \tau \xi.
  \]
  Then
  \begin{multline*}
    \int_\Omega \psi(x) \biggl(\int_{\Sigma\times S^{d-1} \times S^{d-1}} \tau \xi d\rho_x(\tau,\xi,\zeta)\biggr) d(\abs{\mu}+\abs{\nu})\\
    \begin{aligned}
      &= \int_{\Omega \times\Sigma\times S^{d-1} \times S^{d-1}} \varphi(x,\tau,\xi,\zeta) \,d\sigma\\
      &= \lim_k \int_{\Omega \times\Sigma\times S^{d-1} \times S^{d-1}} \varphi(x,\tau,\xi,\zeta) \,d\sigma_k\\
      &= \lim_k \int_\Omega \varphi\bigl(x,t_k(x),m_k(x),n_k(x)\bigr)\,d(\abs{\mu_k}+\abs{\nu_k})\\
      &= \lim_k \int_\Omega \psi(x) t_k(x) m_k(x)\, d(\abs{\mu_k}+\abs{\nu_k}) \\
      &= \lim_k \int_\Omega \psi(x)\, d \mu_k(x) \\
      &= \int_\Omega \psi(x)\, d\mu(x) \\
      &= \int_\Omega \psi(x) t(x) m(x)\, d(\abs{\mu}+\abs{\nu}).
    \end{aligned}
  \end{multline*}
  Since $\psi \in C_0(\Omega)$ was arbitrary, we obtain that
  \[
    \int_{\Sigma \times S^{d-1} \times S^{d-1}} \tau\xi\, d\rho_x(\tau,\xi,\zeta)
    = t(x)m(x)
  \]
  for $(\abs{\mu}+\abs{\nu})$-a.e.~$x$.
  Similarly, using a function $\varphi(x,\tau,\xi,\zeta) = \psi(x) (1-\tau)\zeta$, one obtains that
  \[
    \int_{\Sigma \times S^{d-1} \times S^{d-1}} (1-\tau)\zeta d\rho_x(\tau,\xi,\zeta)
    = (1-t(x)) n(x)
  \]
  for $(\abs{\mu}+\abs{\nu})$-a.e.~$x$.
  Moreover, using a function $\varphi(x,\tau,\xi,\zeta) = \psi(x) \tau$ and recalling that $\abs{\mu_k}\weaksto \abs{\mu}$, we obtain that
  \[
    \int_{\Sigma\times S^{d-1} \times S^{d-1}} \tau d\rho_x(\tau,\xi,\zeta) = t(x)
  \]
  for $(\abs{\mu}+\abs{\nu})$-a.e.~$x$.

  In particular, we obtain that
  \begin{multline*}
    \frac{1}{2}\int_{\Sigma \times S^{d-1} \times S^{d-1}} \tau \abs{\xi-m(x)}^2 + (1-\tau)\abs{\zeta-n(x)}^2\,d\rho_x(\tau,\xi,\zeta) \\
    \begin{aligned}
    &= \int_{\Sigma \times S^{d-1} \times S^{d-1}} \tau (1-\langle \xi, m(x)\rangle)
    + (1-\tau) (1-\langle \zeta, n(x)\rangle\,d\rho_x(\tau,\xi,\zeta)\\
    &= 1 - \langle t(x) m(x),m(x) \rangle - \langle (1-t(x))n(x), n(x) \rangle\\
    &= 0.
    \end{aligned}
  \end{multline*}
  This shows that the measure $\rho_x$ is concentrated on the set
  \[
    \bigl((0,1) \times \{m(x)\} \times \{n(x)\} \bigr) \cup \bigl(\{0,1\} \times S^{d-1} \times S^{d-1}\bigr).
  \]
  As a consequence, we can write
  \[
    \rho_x = T_x \otimes \delta_{m(x)} \otimes \delta_{n(x)} + \delta_0 \otimes A_x + \delta_1 \otimes B_x,
  \]
  where $T_x \in \mathcal{M}(\Sigma)$ satisfies $T_x(\{0,1\}) = 0$,
  and $A_x$, $B_x \in \mathcal{M}(S^{d-1} \times S^{d-1})$.
  Note moreover that
  \[
    T_x(\Sigma) + A_x(S^{d-1}\times S^{d-1}) + B_x(S^{d-1}\times S^{d-1}) = \rho_x(\Sigma\times S^{d-1}\times S^{d-1}) = 1.
  \]

  Since $f$ is lower semi-continuous and bounded, and $\sigma_k \weaksto \sigma$, we now obtain that
  \[
    \begin{aligned}
      \liminf_k F(\mu_k,\nu_k) 
      &= \liminf_k \int_\Omega f\bigl(x,t_k(x),m_k(x),n_k(x)\bigr)\, d(\abs{\mu_k}+\abs{\nu_k}) \\
      &= \liminf_k \int_{\Omega\times \Sigma\times S^{d-1} \times S^{d-1}} f(x,\tau,\xi,\zeta) \,d\sigma_k\\
      &\ge \int_{\Omega \times \Sigma \times S^{d-1} \times S^{d-1}} f(x,\tau,\xi,\zeta)\,d\sigma\\
      &= \int_\Omega \biggl(\int_{\Sigma\times S^{d-1} \times S^{d-1}} f(x,\tau,\xi,\zeta)\, d\rho_x(\tau,\xi,\zeta)\biggr)\,d(\abs{\mu}+\abs{\nu})\\
      &= \int_\Omega \biggl(\int_\Sigma  f(x,\tau,m(x),n(x))\,dT_x(\tau)\biggr)\,d(\abs{\mu}+\abs{\nu})\\
      &\qquad {}+ \int_\Omega \biggl(\int_{S^{d-1}\times S^{d-1}} f(x,0,\xi,\zeta)\,dA_x(\xi,\zeta)\biggr)\,d(\abs{\mu}+\abs{\nu})\\
      &\qquad {}+ \int_\Omega \biggl(\int_{S^{d-1}\times S^{d-1}} f(x,1,\xi,\zeta)\,dB_x(\xi,\zeta)\biggr)\,d(\abs{\mu}+\abs{\nu}).
    \end{aligned}
  \]

  Next we use that $f(x,0,\xi,\zeta)$ and $f(x,1,\xi,\zeta)$ are constant and obtain that
  \[
    \begin{aligned}
      \int_{S^{d-1}\times S^{d-1}} f(x,0,\xi,\zeta)\,dA_x(\xi,\zeta)
      &= f(x,0,m(x),n(x))\,A_x(S^{d-1}\times S^{d-1}),\\
      \int_{S^{d-1}\times S^{d-1}} f(x,1,\xi,\zeta)\,dB_x(\xi,\zeta)
      &= f(x,1,m(x),n(x))\,B_x(S^{d-1}\times S^{d-1}).
    \end{aligned}
  \]
  Moreover, as the mapping $\tau \mapsto f(x,\tau,\xi,\zeta)$ is convex, we can use Jensen's inequality and estimate
  \[
    \int_\Sigma  f(x,\tau,m(x),n(x))\,dT_x(\tau)
    \ge T_x(\Sigma) f\biggl(x,\frac{1}{T_x(\Sigma)}\int_\Sigma \tau\,dT_x(\tau),\xi,\zeta\biggr).
  \]
  Thus we see that
  \begin{multline*}
    \int_{\Sigma\times S^{d-1} \times S^{d-1}} f(x,\tau,\xi,\zeta)\, d\rho_x(\tau,\xi,\zeta)
    \ge A_x(S^{d-1}\times S^{d-1}) f(x,0,m(x),n(x)) \\
    + B_x(S^{d-1}\times S^{d-1}) f(x,1,m(x),n(x))
    + T_x(\Sigma) f\biggl(x,\frac{1}{T_x(\Sigma)}\int_\Sigma \tau\,dT_x(\tau),\xi,\zeta\biggr).
  \end{multline*}
  Now recall that
  \[
    A_x(S^{d-1}\times S^{d-1}) + B_x(S^{d-1}\times S^{d-1})
    + T_x(\Sigma) = 1
  \]
  and
  \begin{multline*}
    0\cdot A_x(S^{d-1}\times S^{d-1})
    + 1 \cdot B_x(S^{d-1}\times S^{d-1})
    + T_x(\Sigma) \frac{1}{T_x(\Sigma)}\int_\Sigma \tau\,dT_x(\tau)
    \\
    = \int_{\Sigma\times S^{d-1} \times S^{d-1}} \tau d\rho_x(\tau,\xi,\zeta) = m(x)
  \end{multline*}
  for $(\abs{\mu}+\abs{\nu})$-a.e.~$x$.
  Thus we can use the convexity of $f$ w.r.t.~$\tau$ and further estimate
  \[
    \int_{\Sigma\times S^{d-1} \times S^{d-1}} f(x,\tau,\xi,\zeta)\, d\rho_x(\tau,\xi,\zeta)
    \ge f\bigl(x,t(x),m(x),n(x)\bigr).
  \]
  Combining these estimates, we see that
  \begin{multline*}
    \liminf_k F(\mu_k,\nu_k) 
    \ge \int_\Omega \biggl(\int_{\Sigma\times S^{d-1} \times S^{d-1}} f(x,\tau,\xi,\zeta)\, d\rho_x(\tau,\xi,\zeta)\biggr)\,d(\abs{\mu}+\abs{\nu}) \\
    \ge \int_\Omega f\bigl(x,t(x),m(x),n(x)\bigr)\,d(\abs{\mu}+\abs{\nu})
    = F(\mu,\nu),
  \end{multline*}
  which concludes the proof.
\end{proof}

\subsection{Relaxation of integral functionals}\label{se:relaxBV}

Let again $\Omega \subset \R^n$ and let $f\colon\bar{\Omega}\times\Sigma\times S^{d-1}\times S^{d-1} \to \R$
be such that the mappings $(\xi,\zeta) \mapsto f(x,0,\xi,\zeta)$ and $(\xi,\zeta)\mapsto f(x,1,\xi,\zeta)$
are constant for every $x$.
Consider moreover the functional $F \colon L^1(\Omega;\R^d)^2 \to \R$,
\[
  F(u,v) = \int_\Omega f\Bigl(x,\frac{\abs{u}}{\abs{u}+\abs{v}},\frac{u}{\abs{u}},\frac{v}{\abs{v}}\Bigr)
  (\abs{u}+\abs{v})\,dx.
\]
In the following result, we will compute the lower semi-continuous extension
of $F$ to $\mathcal{M}(I;\R^d)$ with respect to strict convergence of measures.

\begin{theorem}\label{th:relaxBV}
  Assume that $f\colon \bar{\Omega} \times \Sigma \times S^{d-1} \times S^{d-1} \to \R$
  is continuous and bounded. Define
  \begin{multline*}
    \hat{F}(\mu,\nu) := \inf\Bigl\{ \liminf_k F(u_k,v_k) :
    u_k \mathcal{L}^n \weaksto \mu,\, v_k\mathcal{L}^n \weaksto \nu,\\
    \norm{u_k}_{L^1} \to \abs{\mu}(\Omega),\ \norm{v_k}_{L^1} \to \abs{\nu}(\Omega)\Bigr\}.
  \end{multline*}
  Then
  \[
    \hat{F}(\mu,\nu) = F_c(\mu,\nu) :=
    \int_\Omega f_c\Bigl(x,\frac{d\abs{\mu}}{d(\abs{\mu}+\abs{\nu})},\frac{d\mu}{d\abs{\mu}},\frac{d\nu}{d\abs{\nu}}\Bigr)\, d(\abs{\mu}+\abs{\nu})
  \]
  for every $\mu$, $\nu \in \mathcal{M}(\Omega;\R^d)$,
  where $f_c$ denotes the lower semi-continuous convex hull of $f$ with respect
  to the second variable.
\end{theorem}

\begin{proof}
  In view of Theorem~\ref{th:lsc}, we see that $F_c$ is lower semi-continuous
  with respect to strict convergence in both components, which implies that
  $F_c \le \hat{F}$.
  Thus it is enough to find for each $\mu$, $\nu \in \mathcal{M}(\Omega;\R^d)$
  sequences $u_k\mathcal{L}^n \weaksto \mu$, $v_k\mathcal{L}^n \weaksto \nu$
  with $\norm{u_k}_{L^1}\to\abs{\mu}(\Omega)$ and $\norm{v_k}_{L^1}\to\abs{\nu}(\Omega)$
  and $F(u_k,v_k) \to F_c(\mu,\nu)$.

  Assume first that $\mu = u\mathcal{L}^n$, $\nu = v \mathcal{L}^n$
  with $u$, $v \in C(\bar{\Omega};\R^d)$.
  We can write
  \begin{multline*}
    F_c(u,v) = \int_\Omega f_c\Bigl(x,\frac{\abs{u}}{\abs{u}+\abs{v}},\frac{u}{\abs{u}},\frac{v}{\abs{v}}\Bigr)\,(\abs{u}+\abs{v})\,dx\\
    = \int_\Omega \Bigl(\alpha(x) f\Bigl(x,a(x),\frac{u}{\abs{u}},\frac{v}{\abs{v}}\Bigr) + \beta(x)f\Bigl(x,b(x),\frac{u}{\abs{u}},\frac{v}{\abs{v}}\Bigr)\Bigr)(\abs{u}+\abs{v})\,dx
  \end{multline*}
  for some $0 \le \alpha(x), \beta(x) \le 1$ with $\alpha(x) + \beta(x) = 1$
  and
  \[
    \alpha(x)a(x) + \beta(x)b(x) = \frac{\abs{u(x)}}{\abs{u(x)}+\abs{v(x)}}.
  \]

  Now consider for $k \in \N$ the family of cubes
  $\hat{Q}_i^k = \frac{1}{2^k} \Pi_{r=1}^n [i_r,i_r+1]$,
  $i \in \Z^d$ and define $Q_i^k = \hat{Q}_i^k \cap \Omega$.
  Then we obtain finite partitions
  \[
    \Omega = \bigcup_{i \in I_k} Q_i^k
  \]
  for finite index sets $I_k \subset \Z^n$.
  Choose moreover for all $k$ and $i \in I_k$ some $x_i^k \in Q_i^k$
  and a partition
  $Q_i^k = A_i^k \dot\cup B_i^k$ with disjoint sets $A_i^k$, $B_i^k$
  such that $\abs{A_i^k} = \alpha(x_i^k) \abs{Q_i^k}$ and $\abs{B_i^k} = \beta(x_i^k) \abs{Q_i^k}$.
  
  Define now functions $u_k$, $v_k$ by
  \[
    \begin{aligned}
      u_k(x) &= a(x_i^k)\frac{\abs{u(x_i^k)}+\abs{v(x_i^k)}}{\abs{u(x_i^k)}} u(x_i^k),&
      && x \in A_i^k,\\
      v_k(x) &= (1-a(x_i^k)) \frac{\abs{u(x_i^k)}+\abs{v(x_i^k)}}{\abs{v(x_i^k)}} v(x_i^k),&
      && x \in A_i^k,\\
      u_k(x) &= b(x_i^k) \frac{\abs{u(x_i^k)}+\abs{v(x_i^k)}}{\abs{u(x_i^k)}} u(x_i^k),&
      && x \in B_i^k,\\
      v_k(x) &= (1-b(x_i^k)) \frac{\abs{u(x_i^k)}+\abs{v(x_i^k)}}{\abs{v(x_i^k)}} v(x_i^k),&
      && x \in B_i^k.\\
    \end{aligned}
  \]
  Then
  \[
    \frac{\abs{u_k(x)}}{\abs{u_k(x)} + \abs{v_k(x)}} =
    \begin{cases}
      a(x_i^k), & \text{ if } x \in A_i^k,\\
      b(x_i^k), & \text{ if } x \in B_i^k.
    \end{cases}
  \]
  Moreover we have for all $\ell \in \N$ and $k \ge \ell$ that
  \[
    \begin{aligned}
      \int_{Q_j^\ell} u_k(x)\,dx
      &= \sum_{Q_i^k \subset Q_j^\ell} \frac{\abs{u(x_i^k)} + \abs{v(x_i^k)}}{\abs{u(x_i^k)}} \bigl(a(x_i^k) \abs{A_i^k} + b(x_i^k) \abs{B_i^k}\bigr) u(x_i^k)\\
      &= \sum_{Q_i^k \subset Q_j^\ell} \frac{\abs{u(x_i^k)} + \abs{v(x_i^k)}}{\abs{u(x_i^k)}}
      (\alpha(x_i^k) a(x_i^k) + \beta(x_i^k) b(x_i^k)) u(x_i^k) \abs{Q_i^k}\\
      &= \sum_{Q_i^k \subset Q_j^\ell} u(x_i^k) \abs{Q_i^k}.
    \end{aligned}
  \]
  This shows that $u_k \mathcal{L}^n \weaksto u\mathcal{L}^n$.
  Similarly, we obtain that $v_k \mathcal{L}^n \weaksto v\mathcal{L}^n$.
  Also, we have that $\norm{u_k}_{L^1} \to \norm{u}_{L^1}$ and $\norm{v_k}_{L^1} \to \norm{v}_{L^1}$.

  In addition,
  \[
    \begin{aligned}
      F(u_k,v_k) &= \sum_{i\in I_k} \int_{\Omega_i^k} f\Bigl(x,\frac{\abs{u_k(x)}}{\abs{u_k(x)} + \abs{v_k(x)}},
      \frac{u_k(x)}{\abs{u_k(x)}},\frac{v_k(x)}{\abs{v_k(x)}}\Bigr) (\abs{u_k(x)}+\abs{v_k(x)})\,dx\\
      &= \sum_{i\in I_k} \int_{A_i^k} f\Bigl(x,a(x_i^k),\frac{u(x_i^k)}{\abs{u(x_i^k)}},
      \frac{v(x_i^k)}{\abs{v(x_i^k)}}\Bigr) (\abs{v(x_i^k)}+\abs{u(x_i^k)})\,dx\\
      &\qquad{}+ \sum_{i\in I_k} \int_{B_i^k} f\Bigl(x,b(x_i^k),\frac{u(x_i^k)}{\abs{u(x_i^k)}},
      \frac{v(x_i^k)}{\abs{v(x_i^k)}}\Bigr) (\abs{v(x_i^k)}+\abs{u(x_i^k)})\,dx.
    \end{aligned}
  \]
  Because of the continuity of $f$ and the fact that $\abs{A_i^k} = \alpha(x_i^k)\abs{Q_i^k}$
  and $\abs{B_i^k} = \beta(x_i^k) \abs{Q_i^k}$,
  it follows that
  \[
    \begin{aligned}
      \lim_k F(u_k,v_k)
      &= \int_\Omega \alpha(x) f\Bigl(x,a(x),\frac{u(x)}{\abs{u(x)}},\frac{v(x)}{\abs{v(x)}}\Bigr)\,(\abs{u(x)}+\abs{v(x)})\,dx\\
      &\qquad{} + \int_\Omega \beta(x) f\Bigl(x,b(x),\frac{u(x)}{\abs{u(x)}},\frac{v(x)}{\abs{v(x)}}\Bigr)\,(\abs{u(x)}+\abs{v(x)})\,dx\\
      &= F_c(u,v).
    \end{aligned}
  \]
  This proves the assertion for $u$, $v \in C(\bar{\Omega};\R^d)$.

  Now let $\mu$, $\nu \in \mathcal{M}(\Omega;\R^d)$ be arbitrary.
  Then we can define the regularised measures $\mu_\eps = \varphi_\eps \ast \mu$
  and $\nu_\eps = \varphi_\eps \ast \nu$, where $\varphi_\eps$ is a (scaled)
  standard mollifier.
  Then $\mu_\eps \weaksto \mu$ and $\nu_\eps \weaksto \nu$.
  In addition, we have that $\abs{(\mu_\eps,\nu_\eps)}(\Omega) \to \abs{(\mu,\nu)}(\Omega)$
  (see~\cite[Thm.~2.2]{AmbFusPal00}). As a consequence,
  the Reshetnyak continuity theorem \cite[Thm.~10.3]{Rin18} implies that $F_c(\mu_\eps,\nu_\eps) \to F_c(\mu,\nu)$.
  Now note that $\mu_\eps$ and $\nu_\eps$ are of the form $\mu_\eps = u_\eps \mathcal{L}^n$
  and $\nu_\eps = v_\eps \mathcal{L}^n$ with continuous functions $u_\eps$, $v_\eps$.
  Thus we can use the first part of this proof to find, for every $\eps > 0$, sequences
  $u_\eps^k$, $v_\eps^k$ with $u_\eps^k\mathcal{L}^n \weaksto u_\eps\mathcal{L}^n$,
  $v_\eps^k\mathcal{L}^n \weaksto v\mathcal{L}^n$ and $\norm{u_\eps^k}_{L^1} \to \norm{u_\eps}_{L^1}$,
  $\norm{v_\eps^k}_{L^1} \to \norm{v_\eps}_{L^1}$, such that
  $F(u_\eps^k,v_\eps^k) \to F_c(u_\eps,v_\eps)$.
  By choosing an appropriate diagonal sequence,
  the claim of the theorem is proven.
\end{proof}

\subsection{Proof of Theorem~\ref{th:relax}}\label{se:relaxS}

We will now
apply the results of the previous sections to
the particular functional $F\colon L^1(I;\R^d)^2 \to \R$,
\[
  F(u,v) = -\int_I \inner[B]{\frac{u}{\abs{u}}}{\frac{v}{\abs{v}}}
  \sqrt{\abs{u}\abs{v}}\,dx.
\]
This can be written in the form required by Theorems~\ref{th:lsc} and~\ref{th:relaxBV}
by defining $f \colon \Sigma \times S^{d-1} \times S^{d-1} \to \R$,
\[
  f(t,\xi,\zeta) = -\langle \xi, \zeta \rangle \sqrt{t(1-t)}.
\]
Note that we do not have any dependence on the $x$-variable.

According to Theorem~\ref{th:relaxBV}, we require the convex hull $f_c$
of the function $f$ with respect to the $t$ variable for fixed $\xi$ and $\zeta$.
However, if $\langle \xi,\zeta \rangle \ge 0$, then
the mapping $t \mapsto f(t,\xi,\zeta)$ is already convex and thus
$f(t,\xi,\zeta) = f_c(t,\xi,\zeta)$ in this case.
Conversely, if $\langle \xi,\zeta \rangle < 0$, then
the convex hull of the mapping $t \mapsto f(t,\xi,\zeta)$
is the constant function $f_c(t,\xi,\zeta)$.
This can be summarised to
\[
  f_c(t,\xi,\zeta) = -\langle \xi,\zeta \rangle^+ \sqrt{t(1-t)}
\]
for every $(t,\xi,\zeta) \in \Sigma \times S^{d-1} \times S^{d-1}$.
As a consequence, we have that
\begin{multline*}
  \hat{F}(\mu,\nu) = - \int_I f_c\Bigl(\frac{d\abs{\mu}}{d(\abs{\mu}+\abs{\nu})},\frac{d\mu}{d\abs{\mu}},\frac{d\nu}{d\abs{\nu}}\Bigr)\,d(\abs{\mu}+\abs{\nu})\\
  =
  \int_I \inner[B]{\frac{d\mu}{\abs{d\mu}}}{\frac{d\nu}{\abs{d\nu}}}^+
    \sqrt{\frac{d\abs{\mu}}{d(\abs{\mu}+\abs{\nu})}\frac{d\abs{\nu}}{d(\abs{\mu}+\abs{\nu})}}\,d(\abs{\mu}+\abs{\nu}).
\end{multline*}
As discussed in the beginning of Section~\ref{se:pfthm1}, this 
then implies that
\begin{multline*}
  \hat{S}(c_1,c_2) = - \hat{F}(Dc_1,Dc_2)\\
 = \int_I \inner[B]{\frac{dDc_1}{\abs{dDc_1}}}{\frac{dDc_2}{\abs{dDc_2}}}^+
    \sqrt{\frac{d\abs{Dc_1}}{d(\abs{Dc_1}+\abs{Dc_2})}\frac{d\abs{Dc_2}}{d(\abs{Dc_1}+\abs{Dc_2})}}\,d(\abs{Dc_1}+\abs{Dc_2}),
\end{multline*}
which concludes the proof of Theorem~\ref{th:relax}.

\section{The shape distance on $\BV(I;\R^d)$ and $\SBV(I;\R^d)$}\label{se:pfshape}

\subsection{Distance between reparametrised BV functions}\label{se:relaxD}

In this section, we will prove Propositions~\ref{pr:relaxD} and~\ref{pr:Dinvar}.
To that end, we will need some results concerning the
pointwise convergence of strictly convergent sequences of $\BV$-functions.
The basis for these is the following result on the properties of
strictly convergent measures.

\begin{lemma}
  Let $\mu \in \mathcal{M}(I;\R^d)$ and assume that $\mu_k \weaksto \mu$
  such that $\abs{\mu_k}(I) \to \abs{\mu}(I)$. If $U \Subset I$ is open with
  $\mu(\partial U) = 0$, then $\mu_k(U) \to \mu(U)$.
\end{lemma}

\begin{proof}
  By~\cite[Cor.~10.2]{Rin18} we obtain that $\abs{\mu_k} \weaksto \abs{\mu}$.
  Now we can use~\cite[Cor.~1.204]{FonLeo07} to obtain the assertion.
\end{proof}

\begin{lemma}\label{le:pwconv}
  Let $c \in \BV(I;\R^d)$, and assume that the sequence
  of functions $c_k \in \BV(I;\R^d)$ converges strictly to $c$.
  Then for every $x \in I$ the set of accumulation points of
  the sequence $c_k(x)$ is contained in the interval $[c^\ell(x),c^r(x)]$.
  In particular, if $c$ is continuous at $x$, then $c_k(x) \to c(x)$.
\end{lemma}

\begin{proof}
  Let $z \in \R^d$ be any accumulation point of the sequence $c_k(x)$.
  After possibly passing to a subsequence, we may then assume that
  $c_k(x) \to z$.
  Assume now that $z \not\in [c^\ell(x),c^r(x)]$ and denote
  \[
    \eps := \abs{z-c^\ell(x)} + \abs{z-c^r(x)} - \abs{c^\ell(x)-c^r(x)} > 0.
  \]  
  Since $c_k \to c$ weakly$^*$ in $\BV(I;\R^d)$,
  it follows that $c_k(y) \to c(y)$ for almost every $y \in I$.
  We can therefore choose $x_0 < x$ and $x_1 > x$ such that
  $c_k(x_0) \to c(x_0)$, $c_k(x_1) \to c(x_1)$,
  all the functions $c$ and $c_k$ are continuous at $x_0$ and $x_1$,
  $\abs{Dc}(x_0,x) + \abs{Dc}(x,x_1) < \eps/4$,
  $\abs{c(x_0) - c^\ell(x)} < \eps/4$, and
  $\abs{c(x_1) - c^r(x)} < \eps/4$.
  Then
  \[
    \abs{Dc}(I) = \abs{Dc}(0,x_0) + \abs{Dc}(x_0,x_1) + \abs{Dc}(x_1,1).
  \]
  We have that $\abs{Dc}(0,x_0) = \lim_k \abs{Dc_k}(0,x_0)$ and $\abs{Dc_x}(x_1,1)$.
  Moreover,
  \begin{multline*}
    \abs{Dc}(x_0,x_1)
    = \abs{Dc}(x_0,x) + \abs{Dc}(x,x_1) + \abs{c^\ell(x)-c^r(x)}\\
    < \abs{z-c^\ell(x)} + \abs{z-c^r(x)} - \frac{3\eps}{4}
    < \abs{z-c(x_0)} + \abs{z-c(x_1)} - \frac{\eps}{4}.
  \end{multline*}
  However, we have that
  \begin{multline*}
    \liminf_k \abs{Dc_k}(x_0,x_1) \ge \liminf_k \abs{c_k(x_0)-c_k(x)} + \abs{c_k(x)-c_k(x_1)}\\
    = \abs{z-c(x_0)} + \abs{z-c(x_1)} > \abs{Dc}(x_0,x_1) + \frac{\eps}{4}.
  \end{multline*}
  Combining the results above, we obtain that
  $\liminf_k \abs{Dc_k}(I) > \abs{Dc}(I) + \eps/4$, which contradicts
  the assumption that $c_k$ converges strictly to $c$.
\end{proof}

\begin{lemma}\label{le:simconv}
  Let $c \in \BV(I;\R^d)$, and assume that the sequence
  of functions $c_k \in \BV(I;\R^d)$ converges strictly to $c$.
  Assume moreover that $x \in I$ is such that $c$ is continuous
  at $x$. Then we have for every sequence $x_k \to x$ that $c_k(x_k) \to c(x)$.
\end{lemma}

\begin{proof}
  We can estimate
  \[
    \abs{c_k(x_k)-c(x)} \le \abs{c_k(x_k) - c_k(x)} + \abs{c_k(x)-c(x)}.
  \]
  The function $c$ is continuous at $x$ and thus, in view of Lemma~\ref{le:pwconv},
  the last term tends to zero as $k \to \infty$.
  For the first term, we can estimate
  \[
    \abs{c_k(x_k)-c_k(x)} \le \abs{Dc_k}([x_k,x]).
  \]
  Let now $\eps > 0$ be such that $\abs{Dc}(\{x-\eps,x+\eps\}) = 0$.
  Then we have for sufficiently large $k$ that $[x_k,x] \subset (x-\eps,x+\eps)$
  and thus
  \[
    \limsup_k \abs{Dc_k}([x_k,x])
    \le \limsup_k \abs{Dc_k}(x-\eps,x+\eps)
    = \abs{Dc}(x-\eps,x+\eps).
  \]
  Since this holds for almost every $\eps > 0$ we obtain that
  \[
    \limsup_k \abs{Dc_k}([x_k,x]) \le \abs{Dc}(\{x\}) = 0.
  \]
  Combining all the estimates, we arrive at the claim.
\end{proof}

We now are able to prove Proposition~\ref{pr:relaxD}.

\begin{proof}[Proof of Proposition~\ref{pr:relaxD}]
  Let $c_i \in \BV(I;\R^d)$ and $\varphi_i \in \bar{\Gamma}$ be fixed.
  \medskip
  
  Assume that $\{c_i^{(k)}\}_{k\in\N} \subset \AC(I;\R^d)$ converge
  strictly to $c_i$ and that $\{\varphi_i^{(k)}\}_{k\in\N} \subset \Gamma$
  converge uniformly to $\varphi_i$.
  Since $\len(c_i^{(k)}\circ \varphi_i^{(k)}) = \len(c_i^{(k)}) \to \len(c)$,
  the functions $c_i^{(k)}\circ\varphi_i^{(k)}$ are uniformly bounded in
  $\BV(I;\R^d)$. After possibly passing to a subsequence,
  we may therefore assume that $c_i^{(k)}\circ\varphi_i^{(k)} \weaksto g_i$
  for some $g_i \in \BV(I;\R^d)$.
  Now let $\eps > 0$. Then there exist $0 < x_1 < \ldots < x_N < 1$ such that
  $c_1$ is continuous at $\varphi_1(x_\ell)$ for each $1 \le \ell \le N$, and
  \[
    \len(c_1) \le \sum_{\ell=1}^{N-1} \abs{c_1(\varphi_1(x_{\ell+1})) - c_1(\varphi_1(x_\ell))} + \eps.
  \]
  In addition, we can choose the points $x_\ell$ in such a way
  that $c_1^{(k)}(\varphi_1^{(k)}(x_\ell)) \to g_1(x_\ell)$.
  Since $\varphi_1^{(k)}(x_\ell) \to \varphi_1(x_\ell)$ for all
  $\ell$ and $c_1^{(k)} \to^s c_1$, we obtain from Lemma~\ref{le:simconv} that
  \begin{multline*}
    \len(g_1) \ge \sum_{\ell=1}^{N-1} \abs{g_1(x_{\ell+1})-g_1(x_\ell)}
    = \lim_k \sum_{\ell=1}^{N-1} \abs{c_1^{(k)}(\varphi_1^{(k)}(x_{\ell+1})) - c_1^{(k)}(\varphi_1^{(k)}(x_\ell))}\\
    = \sum_{\ell=1}^{N-1} \abs{c_1(\varphi_1(x_{\ell+1})) - c_1(\varphi_1(x_\ell))}
    \ge \len(c_1) - \eps.
  \end{multline*}
  Since $\eps$ was arbitrary, this shows that $\len(g_1) \ge \len(c_1) = \lim_k \len(c_1^{(k)}\circ\varphi_1^{(k)})$,
  which in turn shows that, actually, $c_1^{(k)}\circ\varphi_1^{(k)}$ converges strictly to $g_1$.
  Similarly, we obtain that $c_2^{(k)}\circ\varphi_2^{(k)} \to^s g_2$.
  Because $\hat{d}$ is strictly lower semi-continuous, it follows that
  \[
    \hat{d}(g_1,g_2) \le \liminf_{k\in\N} \hat{d}(c_1^{(k)}\circ \varphi_1^{(k)},c_2^{(k)}\circ\varphi_2^{(k)})
    \le \liminf_{k\in\N} D(c_1^{(k)},c_2^{(k)};\varphi_1^{(k)},\varphi_2^{(k)}).
  \]
  Next note that it follows from Lemma~\ref{le:pwconv} that $g_i \in [c_i,\varphi_i]$,
  and thus
  \[
    \inf_{g_i \in [c_i,\varphi_i]} \hat{d}(g_1,g_2)
    \le \liminf_{k\in\N} D(c_1^{(k)},c_2^{(k)};\varphi_1^{(k)},\varphi_2^{(k)}).
  \]
  Since the sequences $\{c_i^{(k)}\}_{k\in\N}$ and $\{\varphi_i^{(k)}\}_{k\in\N}$
  converging to $c_i$ and $\varphi_i$, respectively, were arbitrary,  
  and $\hat{D}$ is the lower semi-continuous hull of $D$, it follows that
  \[
    \inf_{g_i \in [c_i,\varphi_i]} \hat{d}(g_1,g_2) \le \hat{D}(c_1,c_2;\varphi_1,\varphi_2).
  \]
  \medskip
  
  Now assume that $g_i \in [c_i,\varphi_i]$ and that
  $\{g_i^{(k)}\}_{k\in\N}\subset\AC(I;\R^d)$ converge strictly
  to $g_i$. Define
  \[
    \varphi_i^{(k)} = \frac{1}{k}\Id + \Bigl(1-\frac{1}{k}\Bigr)\varphi_i,
  \]
  and let
  \[
    c_i^{(k)} := g_i^{(k)} \circ (\varphi_i^{(k)})^{-1}.
  \]
  Then $\varphi_i^{(k)} \in \Gamma$, $c_i^{(k)}\in\AC(I;\R^d)$, and $\varphi_i^{(k)} \to \varphi$ uniformly.
  Moreover, by definition of $c_i^{(k)}$, and since $\len(g_i) = \len(c_i)$,
  we have that $\len(c_i^{(k)}) = \len(g_i^{(k)}) \to \len(g_i) = \len(c_i)$.
  Next, we show that $c_1^{(k)}(y) \to c_1(y)$ at every point $y$ where $c_1$ is continuous.

  To that end, let $y$ be such that $c_1^\ell(y) = c_1^r(y)$.
  Then $c_1^{(k)}(y) = g_1^{(k)}((\varphi_1^{(k)})^{-1}(y))$.
  Since the sequence $(\varphi_1^{(k)})^{-1}(y)$ is bounded, it has a convergent subsequence,
  say $(\varphi_1^{(k')})^{-1}(y) \to z$. Since $\varphi_1^{(k')}$ converges pointwise to $\varphi_1$,
  it follows that $y = \varphi_1(z)$.
  As a consequence, we have that
  $g_1(z) \in [c_1^\ell(\varphi_1(z)),c_1^r(\varphi_1(z))] = [c^\ell(y),c_1^r(y)] = \{c_1(y)\}$,
  which implies in particular that $g_1$ is continuous at $z$.
  Since $g_1^{(k')} \to g_1$ and $(\varphi_1^{(k')})^{-1}(y) \to z$, it follows from Lemma~\ref{le:simconv}
  that
  \[
    c_1^{(k')}(y) = g_1^{(k')}((\varphi_1^{(k')})^{-1}(y)) \to g_1(z) = c_1(y).
  \]
  Since this holds for every convergent subsequence,
  we now obtain that $c_1^{(k)}(y) \to c_1(y)$.
  Together with the convergence $\len(c_1^{(k)}) \to \len(c_1)$,
  this implies that $c_1^{(k)} \to^s c_1$.
  Similarly we obtain that $c_2^{(k)} \to^s c_2$.

  Since $\hat{D}$ is the lower semi-continuous hull of $D$
  with respect to strict convergence in the $c_i$ components
  and uniform convergence in the $\varphi_i$ components, it follows that
  \begin{multline*}
    \hat{D}(c_1,c_2;\varphi_1,\varphi_2)
    \le \liminf_k D(c_1^{(k)},c_2^{(k)};\varphi_1^{(k)},\varphi_2^{(k)})\\
    = \liminf_k d(c_1^{(k)}\circ\varphi_1^{(k)},c_2^{(k)}\circ\varphi_2^{(k)})
    = \liminf_k d(g_1^{(k)},g_2^{(k)}).
  \end{multline*}
  Since $g_i \in [c_i,\varphi_i]$ were arbitrary,
  and $\{g_i^{(k)}\}_{k\in\N} \subset \AC(I;\R^d)$ were arbitrary sequences
  converging strictly to $g_i$ it follows from the definition of $\hat{d}$ that
  \begin{multline*}
    \hat{D}(c_1,c_2;\varphi_1,\varphi_2)
    \le \inf_{g_i\in [c_i,\varphi_i]} \inf_{\substack{\{g_i^{(k)}\}_{k\in\N}\subset \AC(I;\R^d)\\g_i^{(k)}\to^s g_i}}
    d(g_1^{(k)},g_2^{(k)})
    = \inf_{g_i\in [c_i,\varphi_i]} \hat{d}(g_1,g_2),
  \end{multline*}
  which concludes the proof.
\end{proof}

\begin{proof}[Proof of Proposition~\ref{pr:Dinvar}]
  Define
  \[
    \psi^{(k)} := \frac{1}{k}\Id + \Bigl(1-\frac{1}{k}\Bigr)\psi.
  \]
  Then $\lVert\psi^{(k)}-\psi\rVert_\infty \le 2/k$.
  Moreover, $\psi^{(k)}$ satisfies $(\psi^{(k)})' \ge 1/k$ almost everywhere,
  and thus $\psi^{(k)}$ is invertible with inverse
  $\vartheta^{(k)} := (\psi^{(k)})^{-1} \in \Gamma$.

  Now assume that the sequences $\{c_i^{(k)}\}_{k\in\N} \subset \AC(I;\R^d)$ converge strictly to $c_i$
  and that $(\varphi_i^{(k)})_{k\in\N} \subset \Gamma$
  converge uniformly to $\varphi_i$.
  Then
  \[
    \lVert \varphi_i^{(k)}\circ\psi^{(k)} - \varphi_i\circ\psi\rVert_\infty \\
    \le \lVert \varphi_i^{(k)}\circ\psi^{(k)} - \varphi_i\circ\psi^{(k)}\rVert_\infty
    + \lVert \varphi_i\circ\psi^{(k)} - \varphi_i\circ\psi\rVert_\infty.
  \]
  Now the first term on the right hand side converges to zero because
  of the uniform convergence of $\varphi_i^{(k)}$ to $\varphi_i$,
  and the second term converges to zero because of the uniform continuity of $\varphi_i$
  and the uniform convergence of $\psi^{(k)} \to \psi$.
  Thus $\varphi_i^{(k)}\circ\psi^{(k)}$ converges uniformly to $\varphi_i \circ\psi$.
  Since $D$ is lower semi-continous with respect to strict
  convergence in the $c_i$ variables and uniform convergence in the $\varphi_i$ variables,
  and the functional $d$ is invariant under simultaneous reparametrisations,
  it follows that
  \begin{multline*}
    \hat{D}(c_1,c_2;\varphi_1\circ\psi,\varphi_2\circ\psi)
    \le \liminf_k d(c_1^{(k)}\circ\varphi_1^{(k)}\circ\psi^{(k)},c_2^{(k)}\circ\varphi_2^{(k)}\circ\psi^{k})\\
    = \liminf_k d(c_1^{(k)}\circ\varphi_1^{(k)},c_2^{(k)}\circ\varphi_2^{(k)}).
  \end{multline*}
  Since this holds for all such sequences $\{c_i^{(k)}\}_{k\in\N}$ and $\{\varphi_i^{(k)}\}_{k\in\N}$,
  it follows that
  \[
    \hat{D}(c_1,c_2;\varphi_1\circ\psi,\varphi_2\circ\psi)
    \le \hat{D}(c_1,c_2;\varphi_1,\varphi_2).
  \]

  Conversely, assume that $\gamma_i^{(k)} \to \gamma_i :=\varphi_i\circ\psi$ uniformly.
  Then we can similarly
  estimate
  \begin{equation}\label{eq:invarianceh2}
    \lVert \gamma_i^{(k)}\circ\vartheta^{(k)} - \varphi_i\rVert_\infty
    \le \lVert \gamma_i^{(k)}\circ\vartheta^{(k)} - \gamma_i\circ\vartheta^{(k)}\rVert_\infty
    + \lVert \gamma_i\circ\vartheta^{(k)} - \varphi_i\rVert_\infty.
  \end{equation}
  Again, the first term converges to zero because of the uniform convergence
  of $\gamma_i^{(k)}$ to $\gamma_i$.
  For the second term, let $x \in I$ and let $y \in I$ with $\psi(y) = x$.
  Then
  \[
    \lvert\gamma_i(\vartheta^{(k)}(x)) - \varphi_i(x)\rvert
    = \lvert \varphi_i(\psi(\vartheta^{(k)}(\psi(y)))) - \varphi_i(\psi(y))\rvert
  \]
  Since $\lVert\psi-\psi^{(k)}\lVert \le 2/k$, there exists $\eps_k$
  with $\abs{\eps_k} \le 2/k$ such that
  \[
    \psi(\vartheta^{(k)}(\psi(y))) = \psi^{(k)}(\vartheta^{(k)}(\psi(y))) + \eps_k
    = \psi(y) + \eps_k.
  \]
  Thus we have that
  \[
    \abs{\gamma_i(\vartheta^{(k)}(x))-\varphi_i(x)}
    = \abs{\varphi_i(\psi(y)+\eps_k)-\varphi_i(\psi(y))}.
  \]
  Because of the uniform continuity of $\varphi_i$, this implies that
  the second term in~\eqref{eq:invarianceh2} converges to zero,
  and therefore $\lVert\gamma_i^{(k)}\circ\vartheta^{(k)}-\varphi_i\rVert_\infty \to 0$.
  With the same argumentation as before, this now implies that
  \[
    \hat{D}(c_1,c_2,\varphi_1,\varphi_2)\le \hat{D}(c_1,c_2,\varphi_1\circ\psi,\varphi_2\circ\psi).
  \]
  With Proposition~\ref{pr:relaxD}, we now arrive at the assertion.
\end{proof}

\subsection{Proof of Theorem~\ref{th:main2}}\label{se:pfmain2}

We start with showing that the infima in~\eqref{eq:sdistSBV} are attained.

\begin{proposition}\label{pr:existence}
  Assume that $c_1$, $c_2 \in \BV(I;\R^d)$.
  Then the optimisation problem
  \[
    \inf_{(\varphi_1,\varphi_2) \in \bar{\Gamma}} \hat{d}([c_1,\varphi_1],[c_2,\varphi_2])
  \]
  admits a solution $(\bar{\varphi}_1,\bar{\varphi}_2)$.
  Moreover, there exist $\bar{g}_i \in [c_i,\bar{\varphi}_i]$ such that
  \[
    \hat{d}([c_1,\bar{\varphi}_1],[c_2,\bar{\varphi}_2]) = \hat{d}(\bar{g}_1,\bar{g}_2).
  \]
\end{proposition}

\begin{proof}
  We follow the proof of~\cite[Proposition 15]{Bru16}.
  
  Assume that $(\varphi_1^{(k)},\varphi_2^{(k)})_{k\in\N}$ is a minimising
  sequence for $\hat{d}([c_1,\cdot],[c_2,\cdot])$.
  Then there exist $\psi^{(k)} \in \bar{\Gamma}$
  such that $\varphi_i^{(k)} = \gamma_i^{(k)} \circ \psi^{(k)}$, where
  $\gamma_i^{(k)} \in \bar{\Gamma}$ are such that
  $(\gamma_1^{(k)})'(x) + (\gamma_2^{(k)})'(x) = 2$ for a.e.~$x \in I$.
  By Proposition~\ref{pr:Dinvar} we have that
  \[
    \hat{d}([c_1,\varphi_1^{(k)}],[c_2,\varphi_2^{(k)}])
    = \hat{d}([c_1,\gamma_1^{(k)}\circ\psi^{(k)}],[c_2,\gamma_2^{(k)}\circ\psi^{(k)}])
    = \hat{d}([c_1,\gamma_1^{(k)}],[c_2,\gamma_2^{(k)}]).
  \]
  Thus $(\gamma_1^{(k)},\gamma_2^{(k)})$ is a minimising sequence
  as well.
  After replacing $\varphi_i^{(k)}$ by $\gamma_i^{(k)}$, we can thus
  assume without loss of generality that all the functions
  $\varphi_i^{(k)}$ are Lipschitz continuous with Lipschitz constant at most $2$.
  After possibly passing to a sub-sequence, we may therefore assume
  without loss of generality that $\varphi_i^{(k)} \to \bar{\varphi}_i$
  for some $\bar{\varphi}_i \in \bar{\Gamma}$.
  Since by construction the functional
  $(\varphi_1,\varphi_2) \mapsto \hat{d}([c_1,\varphi_1],[c_2,\varphi_2])$ is lower semi-continuous
  with respect to uniform convergence,
  it follows that $(\bar{\varphi}_1,\bar{\varphi}_2)$ is a minimiser of $\hat{d}([c_1,\cdot],[c_2,\cdot])$.

  Now, by definition,
  \[
    \hat{d}([c_1,\bar{\varphi}_1],[c_1,\bar{\varphi}_2])
    = \inf_{g_i \in [c_i,\bar{\varphi}_i]} \hat{S}(g_1,g_2).
  \]
  Since $\hat{S}$ is strictly lower semi-continuous and the sets $[c_i,\bar{\varphi}_i]$
  are compact with respect to strict convergence, the existence of $\bar{g}_1$ and $\bar{g}_2$ follows.
\end{proof}

Next we will show that the functions $\bar{g}_i$ in Proposition~\ref{pr:existence}
can be chosen in $\SBV(I;\R^d)$, if $c_i \in \SBV(I;\R^d)$.
For that, we need two preparatory results.

\begin{lemma}\label{le:SBVsing}
  Assume that $c \in \SBV(I;\R^d)$, that $\varphi \in \bar{\Gamma}$,
  and that $g \in [c,\varphi]$.
  Then the singular part $D^s g$ of the measure $Dg$ is concentrated
  on $\varphi^{-1}(\Sigma(c))$.
\end{lemma}

\begin{proof}
  Write $c = c^{(a)} + c^{(j)}$. Then $g = c^{(a)}\circ \varphi + h$
  for some $h \in [c^{(j)},\varphi]$. We have
  that $c^{(a)}\circ\varphi \in \AC(I;\R^d)$, which implies
  that $D^s c = D^s h$.
  We may therefore assume without loss of generality that $c^{(a)} = 0$,
  $c = c^{(j)}$, and $g \in [c^{(j)},\varphi]$.

  Let $x \in I$ and write $\varphi^{-1}(x) = [a,b]$ with $a \le b$. 
  Then there exists a sequence $a_k \to a^-$ such that  
  $c$ is continuous at each $\varphi(a_k)$.
  In particular, we have that $g(a_k) = c(\varphi(a_k))$ for each $k$.
  Since $g^{\ell}$ and $c^{\ell}$ are left continuous and $\varphi(a_k) \to \varphi(a) = x$,
  it follows that
  $g^{\ell}(a) = \lim_{k \to \infty} g(a_k) = \lim_{k \to \infty} c(\varphi(a_k))
  = c^{\ell}(x)$.
  Similarly, we obtain that $g^{r}(b) = c^{r}(x)$.
  This implies in particular that
  \[
    \abs{Dg}(\varphi^{-1}(x)) \ge \abs{g^{r}(b)-g^{\ell}(a)}
    = \abs{c^{r}(x)-c^{\ell}(x)}
  \]
  for every $x \in I$.
  Thus we have that
  \begin{multline*}
    \abs{Dc}(I) = \abs{Dg}(I) \ge \abs{Dg}(\varphi^{-1}(\Sigma(c))) = 
    \sum_{x \in \Sigma(c)} \abs{Dg}(\varphi^{-1}(x))\\
    \ge \sum_{x \in \Sigma(c)} \abs{c^{r}(x) - c^{\ell}(x)}
    = \abs{D^s c}(I) = \abs{Dc}(I).
  \end{multline*}
  This shows that $\abs{Dg}(\varphi^{-1}(\Sigma(c))) = \abs{Dg}(I)$,
  which in turn implies that $\abs{Dg}(I\setminus\varphi^{-1}(\Sigma(c))) = 0$.
\end{proof}

\begin{lemma}\label{le:SBVhelp}
  Assume that $U \subset \R^d$ is a Borel set, $\nu \in \mathcal{M}_+(U)\setminus\{0\}$
  is a non-trivial positive Radon measure on $U$, and $g \colon U \to \R_{\ge 0}$
  is a non-negative Borel function on $U$ with $\int_U g(x)\,d\nu > 0$.
  Assume moreover that $\mu \in \mathcal{M}_+(U)$
  solves the optimisation problem
  \begin{equation}\label{eq:mumax}
    F(\mu) := \int_U g(x) \sqrt{\frac{d\mu}{d(\mu+\nu)} \frac{d\nu}{d(\mu+\nu)}}\, d(\mu+\nu)
    \to \max_{\substack{\mu \in \mathcal{M}_+(U) \\ \mu(U) = 1}}
  \end{equation}
  Then $\mu \ll \nu$.
\end{lemma}

\begin{proof}
  Decompose $\mu = \mu^a + \mu^s$ with $\mu^a \ll \nu$ and $\mu^s \perp \nu$.
  Then $\mu^s$ is concentrated on the set
  \[
    E := \Bigl\{ x \in U : \lim_{\eps \to 0} \frac{\mu(B_\eps(x))}{\nu(B_\eps(x))} = \infty\Bigr\}.
  \]
  Moreover, we have for $(\mu+\nu)$-a.e.~$x \in E$ that
  \[
    \frac{d\nu}{d(\mu+\nu)}(x) = \lim_{\eps\to 0} \frac{\nu(B_\eps(x))}{\mu(B_\eps(x)) + \nu(B_\eps(x))} = 0.
  \]
  As a consequence, $F(\mu) = F(\mu^a)$.
  Now denote by $u \in L^1(U;\nu)$ the density of $\mu^a$ with respect to $\nu$,
  that is, $\mu^a = u\,\nu$. Then
  \[
    \frac{d\mu}{d(\mu+\nu)} = \frac{u}{u+1}
    \quad\text{ and }\quad
    \frac{d\nu}{d(\mu+\nu)} = \frac{1}{u+1}
  \]
  $(\nu+\mu)$-almost everywhere, and
  \[
    F(\mu) = F(\mu^a)
    = \int_U g(x) \sqrt{\frac{u}{u+1}\frac{1}{u+1}}\,d((u+1)\nu)=\int_U g(x) \sqrt{u}\,d\nu.
  \]
  Now define $\hat{\mu} = \mu^a + \mu^s(U)\,\nu$. Then we obtain similarly that
  \[
    F(\hat{\mu}) = \int_U g(x)\sqrt{u+\mu^s(U)}\,d\nu.
  \]
  Since $\int_U g(x)\,d\nu > 0$ and $\mu$ was assumed to be a maximiser
  of $F$, this implies that $\mu^s = 0$, which proves the assertion.
\end{proof}

\begin{remark}\label{re:int0}
  Let $F$ be as in~\eqref{eq:mumax}, but assume that
  $\int_U g(x)\,d\nu = 0$. Let moreover $\mu \in \mathcal{P}(U)$.
  Then we can again decompose $\mu = \mu^a + \mu^s$ with $\mu^a \ll \nu$
  and $\mu^s \perp \nu$. Moreover, we have is in the proof of Lemma~\ref{le:SBVhelp}
  that $F(\mu) = F(\mu^a)$. However, because $\mu^a \ll \nu$
  and $\int_U g(x)\,d\nu = 0$ we have that
  $F(\mu^a) = 0$.
  Thus the functional $F$ is in this case
  the trivial functional $F(\mu) \equiv 0$.
\end{remark}

\begin{proposition}\label{pr:SBVsols}
  Assume that $c_i \in \SBV(I;\R^d)$ and that
  $\varphi_i \in \bar{\Gamma}$ satisfy $\varphi_1' + \varphi_2' = 2$ almost everywhere.
  Then there exist functions $\bar{g}_i \in [c_i,\varphi_i]$
  such that $\bar{g}_i \in \SBV(I;\R^d)$ and
  \[
    \hat{d}([c_1,\varphi_1],[c_2,\varphi_2]) = \hat{d}(g_1,g_2).
  \]
\end{proposition}

\begin{proof}
  Since the sets $[c_i,\varphi_i]$ are strictly compact
  and $\hat{d}([c_1,\cdot],[c_2,\cdot])$ is strictly lower semi-continuous,
  there exist $g_i \in [c_i,{\varphi}_i]$ such that
  $\hat{d}([c_1,\varphi_1],[c_1,\varphi_2]) = \hat{d}(g_1,g_2)$.
  We will show that it is possible to replace $g_1$ and $g_2$ by
  functions $\bar{g}_1$, $\bar{g}_2 \in \SBV(I;\R^d)$ in such a way
  that $\hat{d}(g_1,g_2) = \hat{d}(\bar{g}_1,\bar{g}_2)$,
  or, equivalently, $\hat{S}(g_1,g_2) = \hat{S}(\bar{g}_1,\bar{g}_2)$.
  
  According to Lemma~\ref{le:SBVsing}, the singular part of
  $Dg_i$ is concentrated on ${\varphi}_i^{-1}(\Sigma(c_i))$.
  Now assume that $y \in \Sigma(c_1)$. Since ${\varphi}_1$ is
  continuous and non-decreasing it follows that ${\varphi}_1^{-1}(y)$ 
  is either a single point or a closed interval.
  Denote now by $R \subset \Sigma(c_1)$ the set of jump points $y$
  of $c_1$ for which ${\varphi}_1^{-1}(y)$ is an non-degenerate interval.
  Then the non-atomic part $D^c g_1$ of $D^s g_1$ is concentrated
  on the set $E := \bigcup_{y \in R} \inn {\varphi}_1^{-1}(y)$,
  since ${\varphi}_1^{-1}(\Sigma(c_1)) \setminus E$ is at most countable.

  Let now $y \in R$ and denote $[a,b] := {\varphi}_1^{-1}(y)$.
  By assumption, the function $g_1$ solves the optimisation problem
  $\max_{g \in [c_1,\varphi_1]}\hat{S}(g,g_2)$.
  Moreover, we have that
  \[
    Dg_1 \Lres (a,b) = v\,\abs{Dg_1} \Lres (a,b)
    \qquad\text{ with }\qquad
    v = \frac{c_1^{r}(y) - c_1^{\ell}(y)}{\abs{c_1^{r}(y)-c_1^{\ell}(y)}} \in \R^d.
  \]
  Now let $\mu \in \mathcal{M}_+(a,b)$ be a positive Radon measure
  satisfying $\mu(a,b) = \abs{Dg_1}(a,b)$.
  Then the function $\hat{g}_1$ defined by
  \[
    \hat{g}_1 = \begin{cases}
      g_1(x) & \text{ if } x \not\in (a,b),\\
      g_1^{\ell}(a) + v \mu(a,x) & \text{ if } x \in (a,b),
    \end{cases}
  \]
  satisfies $\hat{g}_1 \in [c_1,{\varphi}_1]$,
  and $D\hat{g}_1 \Lres (a,b) = v\mu$ and $D\hat{g}_1 \Lres (I\setminus(a,b)) = Dg_1 \Lres (I\setminus(a,b))$.
  Thus $\abs{Dg_1}\Lres(a,b)$ solves the optimisation problem
  \begin{equation}\label{eq:Dg1max}
    \int_{(a,b)} \Bigl\langle v, \frac{dDg_2}{d\abs{Dg_2}}\Bigr\rangle^+
    \sqrt{\frac{d\mu}{d(\mu+\abs{dDg_2})} \frac{d\abs{Dg_2}}{d(\mu + \abs{Dg_2})}}\,d(\mu+\abs{Dg_2}) \to \max
  \end{equation}
  where the maximum is taken over all $\mu \in \mathcal{M}_+(a,b)$
  with $\mu(a,b) =\abs{Dg_1}(a,b)$.

  Assume now that
  $\int_{(a,b)} \langle v,dDg_2/d\abs{Dg_2}\rangle^+ d\abs{Dg_2} > 0$.
  Then we obtain from Lemma~\ref{le:SBVhelp} that
  $\abs{Dg_1} \Lres (a,b) \ll \abs{Dg_2} \Lres (a,b)$.
  Since ${\varphi}_1$ is constant on $[a,b]$ and ${\varphi}_1' + {\varphi}_2' = 2$
  almost everywhere, it follows that ${\varphi}_2'$ is strictly increasing
  on $(a,b)$. Thus ${\varphi}_2^{-1}(\Sigma(c_1)\cap (a,b))$ is an
  at most countable union of single points. Since by Lemma~\ref{le:SBVsing}
  the measure $D^s g_2 \Lres (a,b)$ is concentrated
  on ${\varphi}_2^{-1}(\Sigma(c_1))\cap(a,b)$, it is purely atomic.
  Thus $\abs{D^s g_1} \Lres (a,b)$ is purely atomic as well and therefore
  $\abs{D^c g_1} \Lres (a,b) = 0$.

  Now assume that $\int_{(a,b)} \langle v,dDg_2/d\abs{Dg_2}\rangle^+ d\abs{Dg_2} = 0$.
  As seen in Remark~\ref{re:int0} we have in this case that the
  integral in~\eqref{eq:Dg1max} is equal to zero for all choices of $\mu$.
  We can therefore replace $g_1$ on the interval $(a,b)$ by the function
  \[
    \bar{g}_1(x) = g_1^{\ell}(a) + \frac{x-a}{b-a} (g_1^{r}(b)-g_1^{\ell}(a)),
  \]
  corresponding to a choice of $\mu = s\mathcal{L}^1$
  with $s = \abs{Dg_1}(a,b) /(b-a)$ and have that
  $\hat{S}(\bar{g}_1,g_2) = \hat{S}(g_1,g_2)$.

  Repeating this procedure first for each $y\in R$ and then for
  the function $g_2$, we arrive at the claim.
\end{proof}

\begin{proposition}\label{pr:equalmax}
  Let $c_1$, $c_2 \in \SBV(I;\R^d) \setminus\{0\}$. Then
  \begin{equation}\label{eq:equalmax}
    \min_{(\varphi_1,\varphi_2) \in \bar{\Gamma}} \hat{d}([c_1,\varphi_1],[c_2,\varphi_2])
    = \min_{(\psi_1,\psi_2) \in \bar{\Gamma}} \hat{d}\bigl(G(c_1) \circ \psi_1, G(c_2)\circ \psi_2\bigr).
  \end{equation}
\end{proposition}

\begin{proof}
  We start by recalling the construction of the functions $G(c_i)$
  in Section~\ref{se:shapedist}:
  We set 
  \[
    \xi_i(x) = \frac{\lvert D^s c_i \rvert(0,x)}{2\len(c_i)} + (1-\alpha_i)x
    \qquad\qquad\text{ with }
    \alpha_i := \frac{\lvert D^s c_i\rvert(I)}{2\len(c_i)}.
  \]
  Moreover, we denote by $\zeta_i \colon I \to I$ the non-decreasing
  left inverse of $\xi_i$. Then the functions $G(c_i)$ satisfy
  $G(c_i)(x) = c_i(\zeta_i(x))$
  for $x \not \in \zeta(\Sigma(c_i))$ and $G(c_i)(\xi_i(x)) = c_i(x)$ for $x \not\in \Sigma(c_i)$.
  Moreover, we have $G(c_i) \in [c_i,\zeta_i]$.
  
  Assume now that $(\psi_1,\psi_2) \in \bar{\Gamma}$.
  Define $\varphi_i = \zeta_i \circ \psi_i$. Since $G(c_i) \in [c_i,\zeta_i]$,
  it follows that $G(c_i) \circ \psi_i \in [c_i,\varphi_i]$.
  Thus
  \[
    \hat{d}([c_1,\varphi_1],[c_2,\varphi_2]) \le \hat{d}(G(c_1)\circ\psi_1,G(c_2)\circ\psi_2).
  \]
  Since $(\psi_1,\psi_2)$ was arbitrary, this shows that the inequality $\le$ holds in~\eqref{eq:equalmax}.

  Now assume that the maximum of $\hat{d}([c_1,\cdot],[c_2,\cdot])$ is attained at
  $(\varphi_1,\varphi_2) \in \bar{\Gamma}^2$. In view of the proof
  of Proposition~\ref{pr:existence}, we may assume without
  loss of generality that $\varphi_1'+\varphi_2' = 2$ almost everywhere in $I$.
  Since $c_1$, $c_2 \in \SBV(I;\R^d)$, there exist by
  Proposition~\ref{pr:SBVsols} functions $g_1$, $g_2 \in \SBV(I;\R^d)$
  such that $g_i \in [c_i,\varphi_i]$ and $\hat{d}([c_1,\varphi_1],[c_2,\varphi_2]) = \hat{d}(g_1,g_2)$.

  We will now construct functions $h_1$, $h_2 \in \AC(I;\R^d)$ in such a way that
  $\hat{d}(g_1,g_2) = \hat{d}(h_1,h_2)$.
  The construction is similar as for $G(c_i)$, but we have to be careful to
  keep the function value of $\hat{d}$ unchanged. We denote therefore
  \[
    \beta = \frac{\abs{D^s g_1}(I) + \abs{D^s g_2}(I)}{2(\len(c_1)+\len(c_2))}
  \]
  and define the function $\gamma \colon I \to I$,
  \[
    \gamma(x) = \frac{\abs{D^s g_1}(0,x) + \abs{D^s g_2}(0,x)}{2(\len(c_1)+\len(c_2))} + (1-\beta)x.
  \]
  Next we denote by $\vartheta \colon I \to I$ the non-decreasing
  left inverse of $\gamma$, and define the functions $h_i \in \AC(I;\R^d)$
  by $h_i(x) = g_i(\vartheta(x))$ if $\gamma$ (and thus $g_i$) is continuous at $\vartheta(x)$ and
  \[
    h_i(x) = g_i^{\ell}(\vartheta(x)) + \frac{x-\gamma^{\ell}(\vartheta(x))}{\gamma^{r}(\vartheta(x))-\gamma^{\ell}(\vartheta(x))} \bigl(g_i^{r}(\vartheta(x))-g_i^{\ell}(\vartheta(x))\bigr)
  \]
  else.

  Now assume that $x \in I$ is such that $y := \vartheta(x) \in \Sigma(g_1)\cup\Sigma(g_2)$.
  Then we have by construction of $h_i$ that
  \[
    \dot{h}_i(x) = \frac{g_i^{r}(y)-g_i^{\ell}(y)}{\gamma^{r}(y)-\gamma^{\ell}(y)}
    = \frac{[g_i](y)}{[\gamma](y)}.
  \]
  Thus
  \[
    \begin{aligned}
      \hat{S}(h_1,h_2) &= \int_I \Bigl\langle \frac{\dot{h}_1}{\abs{\dot{h}_1}},\frac{\dot{h}_2}{\abs{\dot{h}_2}}\Bigr\rangle^+
      \sqrt{\abs{\dot{h}_1}\abs{\dot{h}_2}}\,dx \\
      &= \int_{\gamma(I\setminus (\Sigma(g_1)\cup\Sigma(g_2)))}
      \Bigl\langle \frac{\dot{g}_1\circ{\vartheta}}{\abs{\dot{g}_1\circ{\vartheta}}},
      \frac{\dot{g}_2\circ{\vartheta}}{\abs{\dot{g}_2\circ{\vartheta}}}\Bigr\rangle^+
      \sqrt{\abs{\dot{g}_1\circ{\vartheta}}\abs{\dot{g}_2\circ\vartheta}}\,\abs{\vartheta'}\,dx\\
      &\qquad + \sum_{y\in\Sigma(g_1)\cup \Sigma(g_2)} \int_{\gamma^{\ell}(y)}^{\gamma^{r}(y)}
      \Bigl\langle \frac{[g_1](y)}{\abs{[g_1](y)}},\frac{[g_2](y)}{\abs{[g_2](y)}}\Bigr\rangle^+\frac{\sqrt{\abs{[g_1](y)}\,\abs{[g_2](y)}}}{\abs{[\gamma](y)}}\,dx\\
      &= \int_{I\setminus (\Sigma(g_1)\cup\Sigma(g_2))}\Bigl\langle \frac{\dot{g}_1}{\abs{\dot{g}_1}},\frac{\dot{g}_2}{\abs{\dot{g}_2}}\Bigr\rangle^+ \sqrt{\abs{\dot{g}_1}\,\abs{\dot{g}_2}}\,dx\\
      &\qquad + \sum_{y\in\Sigma(g_1)\cup\Sigma(g_2)} \Bigl\langle \frac{[g_1](y)}{\abs{[g_1](y)}},\frac{[g_2](y)}{\abs{[g_2](y)}}\Bigr\rangle^+\sqrt{\abs{[g_1](y)}\,\abs{[g_2](y)}}\\
      &= \hat{S}(g_1,g_2).
    \end{aligned}
  \]
  Since $\len(g_i) = \len(h_i)$, this implies that also $\hat{d}(h_1,h_2) = \hat{d}(g_1,g_2)$.  
  We will next construct functions $\psi_i \in \bar{\Gamma}$
  such that $h_i = G(c_i) \circ \psi_i$.

  We start by defining the function $\hat{\psi}_i := \xi_i \circ \varphi_i \circ \vartheta$.
  Since $\vartheta$ is Lipschitz, $\varphi_i$ is absolutely continuous,
  and $\xi_i \in \SBV(I;\R)$, it follows that $\hat{\psi}_i \in \SBV(I;\R)$ as well.
  Moreover we have that $G(c_i)(\hat{\psi}_i(x)) = h_i(x)$
  for every $x \in I$ such that $\varphi_i \circ \vartheta(x) \not \in \Sigma(c_i)$.
  We may thus define $\psi(x) := \hat{\psi}(x)$ for $x \in I \setminus (\varphi_i\circ\vartheta)^{-1}(\Sigma(c_i))$.
  
  Now let $y \in \Sigma(c_i)$ and denote $[a,b] := (\varphi_i \circ \vartheta)^{-1}(y)$.
  Let moreover $x \in [a,b]$.
  Then $h_i(x)$ lies on the line segment $[c_i^{\ell}(y),c_i^{r}(y)]$,
  that is, we can write
  \[
    h_i(x) = \lambda(x) c_i^{r}(y) + (1-\lambda(x)) c_i^{\ell}(y)
    = c_i^{\ell}(y) + \lambda(x)\,[c_i](y)
  \]
  for some $0 \le \lambda(x) \le 1$.
  Moreover, since $h_i$ is absolutely continuous, it follows that
  the mapping $x \mapsto \lambda(x)$ is absolutely continuous.
  Define now
  \[
    \psi_i(x) = \xi^{\ell}(y) + \lambda(x)\, [\xi_i](y).
  \]
  Then $\zeta_i(\psi_i(x)) = y$ and thus
  \[
    G(c_i)(\psi_i(x))
    = c_i^{\ell}(y) + \frac{\psi_i(x)-\xi^{\ell}(y)}{[\xi_i](y)} [c_i](y)
    = c_i^{\ell}(y) + \lambda(x)\,[c_i](y) = h_i(x).
  \]
  
  By construction we have that $\psi_i \colon I \to I$ is non-decreasing and
  $G(c_i)\circ \psi_i = h_i$.
  Since the restriction of $\psi_i$ to $(\varphi_i \circ\vartheta)^{-1}(y)$
  is absolutely continuous for each $y \in \Sigma(c_i)$
  and the restriction of $\psi_i$ to $I\setminus (\varphi_i\circ\vartheta)^{-1}(\Sigma(c_i))$
  is in $\SBV(I;\R)$, it follows that $\psi_i \in \SBV(I;\R)$ as well.
  Since in addition $\psi_i$ is continuous, it follows that it is actually
  absolutely continuous and therefore contained in $\bar{\Gamma}$.

  This proves the assertion.
\end{proof}

\begin{proof}[Proof of Theorem~\ref{th:main2}]
  By Proposition~\ref{pr:existence}, both of the infima in~\eqref{eq:sdistSBV}
  are attained at some $\bar{\varphi}_i$, $\bar{\psi}_i \in \bar{\Gamma}$.
  Moreover, by Proposition~\ref{pr:equalmax} we have that
  \[
    \inf_{\varphi_i \in \bar{\Gamma}} \hat{d}([c_1,\varphi_1],[c_2,\varphi_2])
    = \inf_{\psi_i \in \bar{\Gamma}} \hat{d}(G(c_1)\circ\psi_1,G(c_2)\circ\psi_2).
  \]
  It remains to show that this is further equal to
  \[
    \hat{d}^S([c_1],[c_2]) := \inf_{g_i \sim c_i} \hat{d}(g_1,g_2).
  \]
  Assume therefore that $g_i \in [c_i]$.
  Denote by $h_i \in \AC(I;\R^d)$ the constant length parametrisation of $G(g_i)$.
  Then we can write $G(g_i) = h_i \circ \vartheta_i$ for some $\vartheta_i \in \bar{\Gamma}$.
  Since $c_i \sim g_i$, it follows that
  $h_i$ is also the constant length parametrisation of $G(c_i)$.
  Now, since $\dot{c}_i \neq 0$ almost everywhere, it follows that
  also $G(c_i)' \neq 0$ almost everywhere, and thus we can write $h_i = G(c_i) \circ \eta_i$
  for some $\eta_i \in \bar{\Gamma}$. Thus $G(g_i) = G(c_i) \circ \eta_i \circ \vartheta_i$
  and thus
  \[
    \hat{d}(g_1,g_2) \ge \inf_{\psi_i \in\bar{\Gamma}} \hat{d}(G(c_1)\circ \psi_1, G(c_2)\circ \psi_2).
  \]
  Since this holds for every $g_i \in [c_i]$ and since $G(c_i) \circ \psi_i \in [c_i]$,
  the assertion follows.
\end{proof}

\section*{Acknowlegdements}

I would like to thank Esten Nicolai W{\o}ien for valuable comments and helpful discussions.

\def\cprime{$'$} \providecommand{\noopsort}[1]{}\def\cprime{$'$}

\end{document}